\documentclass[10pt]{amsart}

\usepackage{amsfonts}
\usepackage{amssymb}
\usepackage{amsmath, amsthm, latexsym, mathbbol, xcolor, mathdots}
\usepackage[all]{xy}
\usepackage[alphabetic]{amsrefs}
 \usepackage{relsize}
\usepackage{hyperref}
\usepackage{graphicx}
\usepackage{tikz-cd, tikz-3dplot} 
\tdplotsetmaincoords{0}{0}

\tikzcdset{scale cd/.style={every label/.append style={scale=#1},
    cells={nodes={scale=#1}}}}
\usetikzlibrary{trees}
\usepackage{pgfplots}
\usepackage{pst-plot}
\pgfplotsset{compat=1.10}
\usepgfplotslibrary{fillbetween}
\usetikzlibrary{patterns}
\usepackage[margin=1in]{geometry}

\def \C{\mathbb{C}}
\def \Z{\mathbb{Z}}
\def \R{\mathbb{R}}

\def \Q{\mathbb{Q}}

\def \E{\mathcal{E}}
\def \F{\mathcal{F}}

\def \O{\mathcal{O}}
\def \X{\mathfrak{X}}
\def \U{\mathfrak{U}}
\def \E{\mathcal{E}}
\def \k{\mathbb{k}}
\def \B{\mathfrak{B}}
\def \m{\mathfrak{m}}
\def \L{\mathcal{L}}

\def \aff{\textup{aff}}
\def \sph{\textup{sph}}

\def \Spec{\operatorname{Spec}}

\def \SL{\operatorname{SL}}
\def \GL{\operatorname{GL}}

\def \rank{\operatorname{rank}}

\def \val{\operatorname{val}}
\def \recc{\operatorname{recc}}

\renewcommand{\tilde}{\widetilde }

\theoremstyle{plain}
\newtheorem{theorem}{Theorem}[section]
\newtheorem{lemma}[theorem]{Lemma}
\newtheorem{proposition}[theorem]{Proposition}
\newtheorem{corollary}[theorem]{Corollary}

\newtheorem{THM}{Theorem}

\theoremstyle{definition}
\newtheorem{example}[theorem]{Example}
\newtheorem{definition}[theorem]{Definition}
\newtheorem{remark}[theorem]{Remark}

\newtheorem{DEFIN}[THM]{Definition}

\begin{document}
\title{Toric vector bundles over a discrete valuation ring and Bruhat-Tits buildings}

\author{Kiumars Kaveh}
\address{Department of Mathematics, University of Pittsburgh,
Pittsburgh, PA, USA.}
\email{kaveh@pitt.edu}

\author{Christopher Manon}
\address{Department of Mathematics, University of Kentucky, Lexington, KY, USA}
\email{Christopher.Manon@uky.edu}

\author{Boris Tsvelikhovskiy} 
\address{Department of Mathematics, University of California, Riverside, CA, USA} 
\email{borist@ucr.edu}


\thanks{The first author is partially supported by National Science Foundation Grant DMS-2101843 and a Simons Collaboration Grant (award number 714052). The second author is partially supported by National Science Foundation Grant DMS-2101911 and a Simons Collaboration Grant (award number 587209).}

\begin{abstract}
We give a classification of rank $r$ torus equivariant vector bundles $\E$ on a toric scheme $\X$ over a discrete valuation ring $\O$, in terms of graded piecewise linear maps $\Phi$ from the fan of $\X$ to the (extended) building of $\GL(r)$. This is an extension of Klyachko's classification of torus equivariant vector bundles on toric varieties over a field on one hand, and Mumford's classification of equivariant line bundles on toric schemes over $\O$ on the other hand.
We also give a simple criterion for equivariant splitting of $\E$ into a sum of toric line bundles in terms of its piecewise linear map. Among other things, this work lays the foundations for study of arithmetic geometry of toric vector bundles.
\end{abstract}

\maketitle

\tableofcontents


\section*{Introduction}
Since their inception in the early 20th century, vector bundles have been in the focus of many important developments in modern geometry. Nevertheless, we still know little about the structure of vector bundles, specially over varieties of dimension greater than $1$. Due to their combinatorial nature, equivariant vector bundles on toric varieties provide a rich class of examples that allows one to have a glimpse at the world of vector bundles on higher dimensional varieties.  

Let $T$ be an $n$-dimensional (split) algebraic torus over a field $K$ and let $X = X_\Sigma$ be a $T$-toric variety associated to a fan $\Sigma$. We recall that a \emph{torus equivariant vector bundle} (or \emph{toric vector bundle} for short) on $X$ {is a vector bundle $\E$ on $X$ together with a linear action of $T$ that lifts that of $X$}. It is well-known that toric line bundles are classified by integral piecewise linear functions on the fan $\Sigma$ of the toric variety. They are central objects that connect geometry of toric varieties with theory of convex polytopes and this connection has proven to be extremely fruitful in algebra, geometry and combinatorics.

Toric vector bundles have been famously classified by Klyachko in the remarkable work \cite{Klyachko} which has been very influential in subsequent development of the theory. His classification is in terms of certain data of compatible filtrations in a finite dimensional vector space $E$ (see Section \ref{sec-prelim-toric-vb}). We point out that the first classification of toric vector bundles goes back to Kaneyama and is in terms of certain data of cocycles (\cite{Kaneyama}). In \cite{KM-tpb, KM-tvb-val} and \cite{Payne-branched-cover}, the Klyachko data of a toric vector bundle has been interpreted as a piecewise linear map from the fan of the toric variety to the space of valuations on the vector space $E$ (the \emph{extended Tits building} of $\GL(E)$).

In this paper, we are concerned with toric schemes over \textit{a discrete valuation ring $\O$}. Basic examples of a discrete valuation ring are $\O=\k[[t]]$, the power series ring (which often is considered in tropical geometry) and $\O=\Z_p$, the ring of $p$-adic integers (which often is considered in arithmetic geometry). Toric schemes over a discrete valuation ring have been classified by Mumford in \cite[\S IV]{Kempf}. In \cite{Gubler}, Gubler extends the classification to non-discrete valuation rings.

The main result of the paper is a combinatorial classification of toric vector bundles $\E$ on a toric scheme over a discrete valuation ring $\O$. When the toric scheme is proper, the classification is in terms of \emph{piecewise affine maps} from the polyhedral complex $\Sigma_1$ of the toric scheme to the \emph{extended Bruhat-Tits building} of $\GL(r)$ where $r$ is the rank of $\E$. This on one hand extends Klyachko's classification and on the other hand extends the classification of toric line bundles on toric schemes due to Mumford (see \cite[\S IV.3(e)]{Kempf}, \cite[Section 3.6]{BPS}). 

 
The present paper can be considered as a continuation of ideas in \cite{KM-tpb} where classifying torus equivariant principal bundles (or \emph{toric principal bundles} for short) on toric varieties over a field $K$ is connected with the theory of Tits buildings of algebraic groups. The main result in \cite{KM-tpb} states that for a reductive group $G$, toric principal $G$-bundles on $X_\Sigma$ are classified by \emph{piecewise linear maps} from the fan $\Sigma$ to the \emph{extended Tits building of $G$} (also called the \emph{cone over the Tits building of $G$}). 

Several applications of the results of the present paper  are already emerging in arithmetic and combinatorial algebraic geometry. 
\begin{itemize}
\item In \cite{BKM} the results of the paper are used to give a combinatorial description of equivariant Chern classes of toric vector bundles over a DVR. We hope this to be useful in advancing our understanding of arithmetic Chern classes.  
\item In \cite{BEST}, the authors define tautological classes of matroids which are certain Chow cohomology classes of the permutahedral toric variety. They use this construction to give geometric interpretations of some important invariants of matroids and prove related conjectures. In a work in progress, we use the classification in the present paper, to define notion of a tropical toric vector bundle over a discrete valuation ring. We then use it to extend the \cite{BEST} construction and define tautological classes of \emph{valuated matroids} and their equivariant Chern classes. We hope this approach to lead to a generalization of vector bundles on metric graphs in \cite{Ulirsch} to higher dimensional tropical varieties. 
\item Extending Klyachko's classification of toric vector bundles over a field, it is desirable and ambitious to ask for a combinatorial classification of equivariant vector bundles on (normal) varieties with an action of a torus (of possibly smaller dimension than that of the variety). In \cite{DGKM}, using the results of the paper as a main ingredient, a combinatorial classification of torus equivariant vector bundles on (normal) varieties with a codimension-one torus action (complexity-one $T$-varieties) is given.
\item Similarly, the results of the paper can be used to give a classification of toric vector bundles over $\Spec(D)$ where $D$ is a Dedekind domain, e.g. the ring of integers of a number field. 
\item By considering toric vector bundles on toric schemes over higher rank local fields, the main classification in the paper can be used to suggest an extension of the notion of Bruhat-Tits building to higher dimensional local fields (cf. \cite{Parshin}).
\end{itemize}

Now we give a more detailed account of the main results of the paper. Let $K$ be a field with discrete valuation $\val: K \to \Z \cup \{\infty\}$ and let $\O$ be its discrete valuation ring. Let $T$ be a (split) torus over $\O$. Recall that a $T$-toric scheme $\X = \X_\Sigma$ over $\Spec(\O)$ is given by a fan $\Sigma$ in the upper half-space $N_\R \times \R_{\geq 0}$ where $N$ denotes the cocharacter lattice of $T$ (Section \ref{subsec-toric-schemes}).

To make the results easier to state, first we state them in the case where $\X = \X_\Sigma$ is proper over $\Spec(\O)$, or equivalently $\Sigma$ is complete, that is, the support of $\Sigma$ is the whole $N_\R \times \R_{\geq 0}$. We denote by $\Sigma_1$ (respectively $\Sigma_0$) the polyhedral complex (respectively the fan) obtained by intersecting the cones in $\Sigma$ with the hyperplane $N_\R \times \{1\}$ (respectively $N_\R \times \{0\}$). The fan $\Sigma_0$ is the \emph{recession fan} of the polyhedral complex $\Sigma_1$. The generic fiber $\X_\eta$ is the toric variety (over $K$) associated to the fan $\Sigma_0$. 

\begin{DEFIN}
\label{def-intro-add-norm}
Let $E$ be an $r$-dimensional vector space over $K$. For any $m \geq 0$, a \emph{level $m$ additive norm} on $E$ is a function $v: E \to \overline{\R} = \R \cup \{\infty\}$ satisfying the following axioms: 
\begin{itemize}
\item[(1)] $v(\lambda e) = m \val(\lambda) + v(e)$, for all $e \in E$ and $\lambda \in K$. 
\item[(2)](Non-Archimedean property) $v(e_1+e_2) \geq \min\{v(e_1), v(e_2)\}$, for all $e_1, e_2 \in E$.
\item[(3)] $v(e) = \infty$ if and only if $e = 0$. 
\end{itemize}
We will be mainly concerned with the cases where $m=0, 1$. When $m=1$ we have $v(\lambda e) = \val(\lambda) + v(e)$, and when $m=0$, we have $v(\lambda e) = v(e)$, for all $\lambda \neq 0$. 


We denote the space of all level $m$ additive norms by $\tilde{\B}_m(E)$. In particular, we denote the space of all level $1$ additive norms on $E$ by $\tilde{\B}_\aff(E) = \tilde{\B}_1(E)$. We also denote the space of all level $0$ additive norms by $\tilde{\B}_\sph(E) = \tilde{\B}_0(E)$. We call $\tilde{\B}_\aff(E)$ (respectively $\tilde{\B}_\sph(E)$), the \emph{extended Bruhat-Tits building} (respectively the \emph{extended Tits building}) of $E$. See Section \ref{sec-add-norms}, and also Section \ref{sec-building} for the connection with Tits and Bruhat-Tits buildings of $\GL(E)$. Finally we let $$\tilde{\B}(E) = \bigsqcup_{m \geq 0} \tilde{\B}_m(E),$$
\color{black}{}
and call it the \emph{total extended building} of $E$.
\end{DEFIN}

A choice of a basis $B$ for the vector $K$-space $E$, specifies a subset of $\tilde{\B}_\sph(E)$ which is an $r$-dimensional real vector space where $r =\dim(E)$.
We denote this subset by $\tilde{A}_\sph(E)$ and call it the \emph{extended apartment} associated to the basis $B$. Similarly, we have a subset of $\tilde{\B}_\aff(E)$ which is an $r$-dimensional real affine space. We denote it by $\tilde{A}_\aff(E)$ and call it the \emph{extended apartment} associated to the basis $B$.
Both $\tilde{\B}_\sph(E)$ and $\tilde{\B}_\aff(E)$ are unions of their apartments. 


Let $\Sigma_0$ and $\Sigma_1$ be as above.
In Section \ref{subsec-graded-PL-map}, we define the following notions:
\begin{itemize}
    \item A \emph{piecewise linear map} $\Phi_0$ from a fan $\Sigma_0$ to the extended Tits building $\tilde{\B}_\sph(E)$. Roughly speaking it means that $\Phi_0$ maps each cone $\sigma \in \Sigma_0$ to an extended apartment and restriction to $\sigma$ is a linear map. 
    \item A \emph{piecewise affine map} $\Phi_1$ from a polyhedral complex $\Sigma_1$ to the extended Bruhat-Tits building $\tilde{\B}_\aff(E)$. Roughly speaking it means that $\Phi_1$ maps each polyhedron $\Delta \in \Sigma_1$ to an extended apartment and restriction to $\Delta$ is an affine map.
\end{itemize}

Recall that given a piecewise affine function $\phi: |\Sigma_1| \to \R$, one can define $\phi_0: |\Sigma_0| \to \R$, the \emph{linear part of $\phi_0$}. Here $\Sigma_0$ is the collection of recession cones of the polyhedra in $\Sigma_1$. More generally, for a piecewise affine map $\Phi$, we can define $\Phi_0: |\Sigma_0| \to \tilde{\B}_\sph(E)$, the \emph{linear part of $\Phi$}. 

There is a natural notion of morphism of piecewise affine maps (Definition \ref{def-morphism-pa-map}) and piecewise affine maps to extended Bruhat-Tits buildings of finite dimensional vector spaces $E$ form a category. 

Let $\X_\Sigma$ be a toric scheme over $\Spec(\O)$ associated to a fan $\Sigma$ in $N_\R \times \R_{\geq 0}$ with corresponding polyhedral complex $\Sigma_1$ in $N_\R \times \{1\}$. 
Given a toric vector bundle $\E$ over $\X_\Sigma$, we denote the fiber $\E_{x_0}$, over the distinguished point $x_0$, by $E$. Our main result is the following.
\begin{THM}[Main theorem, first part] \label{th-main-intro}
With notation as above, let $\X_\Sigma$ be a proper toric scheme over $\Spec(\O)$. Then there is a bijection between the isomorphism classes of toric vector bundles on $\X_\Sigma$ and the piecewise affine maps from $|\Sigma_1|$ to $\tilde{\B}_\aff(E)$, for finite dimensional $K$-vector spaces $E$. Moreover, this bijection extends to an equivalence of the corresponding categories. 
\end{THM}

The moral of Theorem \ref{th-main-intro} is that the extended Bruhat-Tits building $\tilde{\B}_\aff(E)$ can be considered as an analogue of the classifying space of $\GL(E)$ for rank $r$ toric vector bundles over proper toric schemes. 

\begin{THM}[Main theorem, second part]  \label{th-main2-intro}
Let $\E$ be a toric vector bundle over a proper toric scheme $\X_\Sigma$ with the corresponding piecewise affine map $\Phi: |\Sigma_1| \to \tilde{\B}_\aff(E)$. Then the generic fiber $\E_\eta$, which is a (usual) toric vector bundle over the toric variety $\X_\eta = X_{\Sigma_0}$ over the field $K$, corresponds to $\Phi_0: |\Sigma_0| \to \tilde{\B}_\sph(E)$, the linear part of 
$\Phi$, which is a piecewise linear map.  
\end{THM}

Theorems \ref{th-main-intro}
and \ref{th-main2-intro} extend the well-known results on classification of torus equivariant line bundles on a toric scheme (\cite[Section 3.3]{BPS}).

When $\X_\Sigma$ is not proper over $\Spec(\O)$ (equivalently, support of $\Sigma$ is not the whole $N_\R \times \R_{\geq 0}$) the main theorems still work but beside the piecewise affine map $\Phi_1: |\Sigma_1| \to \tilde{\B}_\aff(E)$ we also need the data of the piecewise linear map $\Phi_0: |\Sigma_0| \to \tilde{\B}_\sph(E)$ which is compatible with $\Phi_1$. It is convenient to encode the data of $\Phi_1$ and $\Phi_0$ in the notion of a \emph{graded piecewise linear map} $\Phi$ from the fan $\Sigma$ to the total extended building $\tilde{\B}(E) = \bigsqcup_{m \geq 0} \tilde{\B}_m(E)$. We say that a map $\Phi$ is a graded pieceiwse linear map if it is piecewise linear in the sense of Definition \ref{def-PL-map}, and moreover for any $m \geq 0$, it sends every point $(x, m) \in |\Sigma|$ to $\tilde{\B}_m(E)$ (see Definition \ref{def-graded-PL-map}).


\begin{THM}[Main theorem, general case] Suppose $\Sigma$ is a fan in $N_\R \times \R_{\geq 0}$ that is not necessarily complete. Then there is an equivalence of categories between the category of toric vector bundles on the toric scheme $\X_\Sigma$ and the category of graded pieceiwse linear maps $\Phi: |\Sigma| \to \tilde{\B}(E)$ for finite dimensional $K$-vector spaces $E$. Moreover, the toric vector bundle $\E_\eta$ over the generic point corresponds to the piecewise linear map $\Phi_0 := \Phi_{|\Sigma_0}: |\Sigma_0| \to \tilde{\B}_\sph(E)$. 
\end{THM}

In Section \ref{sec-splitting} we give a criterion for equivariant splitting of a toric vector bundle on a toric scheme $\X_\Sigma$. It is an extension of Klyachko's criterion for equivariant splitting of toric vector bundles over a field to the setting of toric schemes over a DVR. We use this criterion to give an example of a rank $2$ toric vector bundle on the canonical model of $\mathbb{P}^1$ over the $p$-adic integers $\Z_2$ that does not split (Example \ref{ex-split-tvb}). 

As Klyachko shows, the splitting of low rank toric vector bundles over projective space (over a field) follows from a Helly-type theorem for the Tits building of $\GL(E)$ (\cite[Theorem 6.1.2]{Klyachko}). In Example \ref{ex-Helly-GL2}, we observe that an analogue of Klyachko's Helly's theorem holds for the Bruhat-Tits building of $\GL(2, \Q_p)$ (which is a $(p+1)$-regular infinite tree). In a future work we extend Klyachko's Helly's theorem to the Bruhat-Tits building of $\GL(E)$, for any finite dimensional $K$-vector space $E$. 
We end the introduction by proposing some problems for further study:
\begin{itemize}
    \item Give a combinatorial characterization of ampleness of a toric vector bundle over a DVR in terms of its piecewise affine map, analogous to \cite{HMP, KM-tvb-val}.
    \item Construct analogue of parliament of polytopes (\cite{DJS}) for a toric vector bundle over a DVR capturing the space of global sections and give a criterion for global generation. The paper \cite{KU} which concerns the line bundle case is relevant to this.
    \item Give criterion for (semi)stability of a toric vector bundle over a DVR in terms of its piecewise affine map.
\end{itemize}

\bigskip
\noindent{\bf Acknowledgement:} We would like to thank Huayi Chen, Martin Sombra, Eric Katz and specially Sam Payne and Jose Burgos Gil for very useful email correspondence. In particular, Jose Burgos Gil suggested to us to work with the space of level $m$ additive norms. The first two authors are supported by an NSF collaborative research grant (DMS-2101911 and DMS-2101843) as well as Simons Collaboration Grants for Mathematicians.

\vspace{.3cm}
\noindent{\bf Notation:}
\begin{itemize}
\item $K$, a discretely valued field with valuation $\val: K \to \overline{\Z} = \Z \cup \{\infty\}$
\item $\O$, the valuation ring of $(K, \val)$
\item $\mathfrak{m}$, the maximal ideal in $\O$
\item $\varpi$, a uniformizer for $\O$ i.e. a generator of $\mathfrak{m}$ 
\item $\k = \O/\mathfrak{m}$, the residue field of $\O$
\item $\Spec(\O)$, the affine scheme associated to $\O$
\item $\eta$, the generic point of $\Spec(\O)$, i.e. the point associated to the prime ideal $\{0\}$
\item $s$, the special point of $\Spec(\O)$, i.e. the point associated to the maximal ideal $\mathfrak{m}$
\item For a scheme $\X$ over $\Spec$, we denote the fibers of $\X$ over $\eta$ and $s$ by $\X_\eta$ and $\X_s$ respectively.
\item $T \cong \mathbb{G}_m^n$, a split torus over $\O$. We have $T_\eta \cong (K^\times)^n$ and $T_s \cong (\k^\times)^n$.
\item $N$ and $M$, cocharacter and character lattices of $T$, with $N_\R = N \otimes \R$ and $M_\R = M \otimes \R$
\item $\tilde{N} = N \times \Z$ and $\tilde{M} = M \times \Z$ 
\item 
For a cone $\sigma \subset N_\R \times \R_{\geq 0}$, we let $\Delta =  \sigma \cap (N_\R \times \{1\})$. Conversely, for a polyhedron $\Delta \subset N_\R \times \{1\}$ we let $\sigma \subset N_\R \times \R_{\geq 0}$ be the cone over $\Delta$.
\item $\Sigma$, a fan in $N_\R \times \R_{\geq 0}$
\item $\Sigma_1$, a polyhedral complex in $N_\R \times \{1\}$ obtained by intersecting the cones in $\Sigma$ with $N_\R \times \{1\}$
\item $\X = \X_{\Sigma}$, toric scheme associated to a fan $\Sigma$ in $N_\R \times \R_{\geq 0}$.
\item $\tilde{\B}_m(E)$, the collection of level $m$ additive norms on a $K$-vector space $E$ (see Definition \ref{def-level-m-add-norm}). We call the union $\tilde{\B}(E) = \bigsqcup_{m \geq 0} \tilde{\B}_m(E)$, the \emph{extended total Bruhat-Tits building} of $E$. 
\item In particular, $\tilde{\B}_\aff(E) = \tilde{\B}_1(E)$ is the collection of all additive norms on a $K$-vector space $E$ (see Definition \ref{def-add-norm}). It is called the \emph{extended Bruhat-Tits building} of $\GL(E)$. 
\item Also, $\tilde{\B}_\sph(E) = \tilde{\B}_0(E)$ is the collection of all valuations on a $K$-vector space $E$ (see Definition \ref{def-val}). We call it the \emph{extended Tits building} of $\GL(E)$. 
\end{itemize}

\section{Preliminaries on toric varieties and toric schemes}   \label{sec-prelim}
\subsection{Klyachko's classification of toric vector bundles over a field}   \label{sec-prelim-toric-vb}
In this section we review Klyachko's classification of toric vector bundles in \cite{Klyachko}.
Let $T \cong \mathbb{G}_m^n$ denote an $n$-dimensional algebraic torus over a field $\k$. We let $M$ and $N$ denote its character and cocharacter lattices respectively. We also denote by $M_\R$ and $N_\R$ the $\R$-vector spaces spanned by $M$ and $N$. For a cone $\sigma \in N_\R$ let $M_\sigma$ be the quotient lattice:
$$M_\sigma = M / (\sigma^\perp \cap M).$$
Let $\Sigma$ be a (finite rational polyhedral) fan in $N_\R$ and let $X_\Sigma$ be the corresponding toric variety. Also $X_\sigma$ denotes the invariant affine open subset in $X_\Sigma$ corresponding to a cone $\sigma \in \Sigma$. We denote the support of $\Sigma$, that is the union of all the cones in $\Sigma$, by $|\Sigma|$. For each $i$, $\Sigma(i)$ denotes the subset of $i$-dimensional cones in $\Sigma$. In particular, $\Sigma(1)$ is the set of rays in $\Sigma$. For each ray $\rho \in \Sigma(1)$ we let $\nu_\rho$ be the primitive vector along $\rho$, i.e. $\nu_\rho$ is the unique vector on $\rho$ whose integral length is equal to $1$.

We say that $\E$ is a {\it toric vector bundle} on $X_\Sigma$ if $\E$ is a vector bundle on $X_\Sigma$ equipped with a $T$-linearization. This means that there is an action of $T$ on $\E$ which lifts the $T$-action on $X_\Sigma$ such that the action map $\E_x \to \E_{t \cdot x}$ for any $t \in T$, $x \in X_\Sigma$ is linear. 

We fix a point $x_0 \in X_0 \subset X_\Sigma$ in the dense orbit $X_0$. We often identify $X_0$ with $T$ and think of $x_0$ as the identity element in $T$. We let $E = \E_{x_0}$ denote the fiber of $\E$ over $x_0$. It is an $r$-dimensional vector space where $r = \rank(\E)$. 

For each cone $\sigma \in \Sigma$ we have an invariant open subset $X_\sigma \subset X_\Sigma$. The space of sections $H^0(X_\sigma, \E)$ is a $T$-module. We let  $H^0(X_\sigma, \E)_u \subseteq H^0(X_\sigma, \E)$ be the weight space corresponding to a weight $u \in M$; these spaces define the weight decomposition: 

$$H^0(X_\sigma, \E) = \bigoplus_{u \in M} H^0(X_\sigma, \E)_u.$$ Every section in $H^0(X_\sigma, \E)_u$ is determined by its value at $x_0$.  Thus, by restricting sections to $E = \E_{x_0}$, we get an embedding $H^0(X_\sigma, \E)_u \hookrightarrow E$. Let us denote the image of $H^0(X_\sigma, \E)_u$ in $E$ by $E_u^\sigma$. Note that if $u' \in \sigma^\vee \cap M$ then \textit{multiplication by the character}  $\chi^{u'}$ 
gives an injection $H^0(X_\sigma, \E)_u \hookrightarrow H^0(X_\sigma, \E)_{u-u'}$. Moreover, the multiplication map by $\chi^{u'}$ commutes with the evaluation at $x_0$ and hence induces an inclusion $E_u^\sigma \subset E_{u-u'}^\sigma$. If $u' \in \sigma^\perp$ then these maps are isomorphisms and thus $E_u^\sigma$ depends only on the class $[u] \in M_\sigma = M/(\sigma^\perp \cap M)$. For a ray $\rho \in \Sigma(1)$ we write $$E^\rho_j = E_u^\rho,$$ for any $u \in M$ with $\langle u, \nu_\rho \rangle = j$ (all such $u$ define the same class in $M_\rho$). Equivalently, one can define $E^\rho_u$ as follows (see \cite[\S 0.1]{Klyachko}). Pick a point $x_\rho$ in the orbit $O_\rho$ and let:
$$E^\rho_u = \{ e \in E \mid \lim_{t \cdot x_0 \to x_\rho} \chi^u(t)^{-1}(t \cdot e) \textup{ exists in } \E \},$$
where $t$ varies in $T$ in such a way that $t \cdot x_0$ approaches $x_\rho$. One checks that $E^\rho_u$ does not depend on the choice of $x_\rho$ and only depends on $j = \langle u, \nu_\rho \rangle$. 

We thus have a decreasing filtration on $E$:
\begin{equation}  \label{equ-filt-E-rho}
\cdots \supset E^\rho_{j-1} \supset E^\rho_j \supset E^\rho_{j+1} \supset \cdots
\end{equation}

An important step in the classification of toric vector bundles is that a toric vector bundle over an affine toric variety is {\it equivariantly trivial}. That is, it decomposes $T$-equivariantly as a sum of trivial line bundles. Let $\sigma$ be a strictly convex rational polyhedral cone with corresponding affine toric variety $X_\sigma$. Given $u \in M$, let $\L_u$ be the trivial line bundle $X_\sigma \times \mathbb{A}^1$ on $X_\sigma$ where $T$ acts on $\mathbb{A}^1$ via the character $-u$. One observes that the toric line bundle $\L_u$ in fact only depends on the class $[u] \in M_\sigma$. Hence we also denote this line bundle by $\L_{[u]}$. One has the following (\cite[Proposition 2.1.1]{Klyachko}):
\begin{proposition}   \label{prop-toric-vb-over-affine-equiv-trivial}
Let $\E$ be a toric vector bundle of rank $r$ on an affine toric variety $X_\sigma$. Then $\E$ splits equivariantly into a sum of line bundles $\L_u$: 
$$\E = \bigoplus_{i=1}^r \L_{[u_i]}$$
where $[u_i] \in M_\sigma$.
\end{proposition}

We denote the multiset $\{ [u_1], \ldots, [u_r]\} \subset M_\sigma$ by $u(\sigma)$. The above shows that the filtrations $(E^\rho_i)_{i \in \Z}$, $\rho \in \Sigma(1)$, satisfy the following compatibility condition: 
There is a decomposition $E = \bigoplus_{j=1}^r L_j$ of $E$ into a direct sum of $1$-dimensional subspaces $L_j$ and a multiset $u(\sigma) = \{ [u_1], \ldots, [u_r]\} \subset M_\sigma$ such that for any ray $\rho \in \sigma(1)$ we have:
\begin{equation}  \label{equ-Klyachko-comp-condition}
E^\rho_i = \sum_{\langle u_j, \nu_\rho \rangle \geq i}  L_j.
\end{equation}

{We call a collection of decreasing $\Z$-filtrations $\{(E_i^\rho) \mid \rho \in \Sigma(1)\}$ a \emph{compatible collection of filtrations} if for any $\sigma \in \Sigma$ there is a direct sum decomposition $E=\bigoplus_{j=1}^r L_j$ of $E$ into $1$-dimensional subspaces and a multiset $\{ [u_1], \ldots, [u_r]\} \subset M_\sigma$ such that \eqref{equ-Klyachko-comp-condition} holds. We also need the notion of a morphism between compatible filtrations. Let $E$, $E'$ be finite dimensional vector spaces with compatible collections of filtrations $\{(E^\rho_i) \mid \rho \in \Sigma(1)\}$ and 
$\{({E'}^\rho_i) \mid \rho \in \Sigma(1)\}$ respectively. A \emph{morphism} between these compatible collections is a linear map $F: E \to E'$ such that $F(E_i^\rho) \subset {E'}_i^\rho$, for all $i \in \Z$ and $\rho \in \Sigma(1)$.} 



The following is Klyachko's result on classification of toric vector bundles (see \cite[Theorem 2.2.1]{Klyachko}).
\begin{theorem}[Klyachko]   \label{th-Klyachko}
The category of toric vector bundles $\E$ on $X_\Sigma$ is naturally equivalent to the category of compatible collections of filtrations on finite dimensional vector spaces $E$. 
\end{theorem}

\subsection{Toric schemes over a discrete valuation ring}   \label{subsec-toric-schemes}
In this section we review classification of toric schemes over a discrete valuation ring. Our main references are \cite[Chapter IV, \S 3]{Kempf} and \cite[Chapter 3]{BPS}.

Let $\val: K \to \overline{\Z}$ be a discrete valuation on a field $K$, and let $\O$ denote the corresponding discrete valuation ring. We let $\varpi$ denote a uniformizer for $\O$, namely an ideal generator for $\mathfrak{m}$, and $\k = \O / \mathfrak{m}$ the residue field of $\O$. Also $s$ and $\eta$ denote the \emph{special} and \emph{generic} points in $\Spec(\O)$ corresponding to the maximal ideal $\mathfrak{m}$ and the prime ideal $\{0\}$ respectively. For a scheme $\X$ over $\Spec(\O)$ we let $\X_\eta = \X \times_{\Spec(\O)} \Spec(K)$ and $\X_s = \X \times_{\Spec(\O)} \Spec(\k)$ be the \emph{generic fiber} and \emph{special fiber} of $\X$ respectively. 

We let $N$ and $M$ be dual free abelian groups of rank $n$, $\tilde{N} = N \times \Z$ and $\tilde{M} = M \times \Z$. 
We also let $T$ denote the split torus over $\Spec(\O)$ with cocharacter lattice $N$. For $u \in M$ we denote the corresponding character of $T$ by $\chi^u$. The generic fiber $T_\eta$ is the split torus over $K$ isomorphic to $(K^\times)^n$. The coordinate ring
$K[T_\eta]$ is isomorphic to the Laurent polynomial algebra with coefficients in $K$ in $n$ indeterminates. 

A \textit{toric scheme} $\X$ is an irreducible normal scheme of finite type over $\Spec(\O)$ together with a $T$-action and a choice of an open embedding of $T_\eta$ in $\X$ such that the action of $T$ on $\X$ extends the translation action of $T_\eta$ on itself. 
We denote the image of the identity point in $T_\eta$ in $\X$ with $x_0$. Note that the generic fiber $\X_\eta$ is a usual toric variety for the torus $T_\eta$ (defined over $K$). On the other hand, the special fiber $\X_s$ is a union of toric varieties for the torus $T_s$ (defined over the residue field $\k$). We think of $\X$ as an ``infinitesimal degeneration'' of the toric variety $\X_\eta$ to $\X_s$, a (possibly non-reduced) union of toric varieties.

For a rational polyhedral cone $\sigma \subset N_\R \times \R_{\geq 0}$ we let $\Delta$ be the polyhedron obtained by intersecting $\sigma$ and the plane $N_\R \times \{1\}$ (it may be empty). Conversely, for a rational polyhedron $\Delta \subset N_\R \times \{1\}$, we let $\sigma \subset N_\R \times \R_{\geq 0}$ denote the cone over $\Delta$. Then $\sigma \leftrightarrow \Delta$ gives a one-to-one correspondence between rational polyhedra in $N_\R \times \{1\}$ and rational polyhedral cones in $N_\R \times \R_{\geq 0}$ that intersect $N_\R \times \{1\}$.

\begin{definition}  \label{def-sigma-Delta}
Let $\Sigma$ be a fan in $N_\R \times \R_{\geq 0}$. We denote by $\Sigma_1$ the polyhedral complex in $N_\R \times \{1\}$ obtained by intersecting all the cones in $\Sigma$ by $N_\R \times \{1\}$. Conversely, for a polyhedral complex $\Sigma_1$ in $N_\R \times \{1\}$ we let $\Sigma$ be the fan consisting of the cones over the polyhedra in $\Sigma$ as well as the faces of these cones. Also we denote by $\Sigma_0$ the fan consisting of cones in $\Sigma$ that lie in $N_\R \times \{0\}$. 
\end{definition}

We recall that the \emph{recession cone} $\recc(\Delta)$ of a polyhedron $\Delta \subset N_\R$ is defined as:
$$\recc(\Delta) = \{ x \in N_\R \mid x+\Delta \subset \Delta \}.$$ 

\begin{remark}   \label{rem-rec-fan}
Let $\Sigma_1$ be a complete strongly convex rational polyhedral complex in $N_\R \times \{1\}$ (for definitions see \cite[Section 2.1]{BPS}). Then the collection $\Sigma$ consisting of cones over the polyhedra in $\Sigma_1$ and their faces is indeed a fan in $N_\R \times \R_{\geq 0}$. In particular, $\Sigma_0$ consists of the recession cones of the polyhedra in $\Sigma_1$ and is a complete fan in $N_\R = N_\R \times \{0\}$. On the other hand, as is shown in \cite[Example 3.1]{BS}, this is not the case if $\Sigma_1$ is not a complete strongly convex rational polyhedral complex. 
\end{remark}

From now on, by a polyhedral complex we always mean a strongly convex polyhedral complex. 

For a strictly convex rational polyhedral cone $\sigma \subset N_\R \times \R_{\geq 0}$ let $R_{\sigma} \subset K[T_\eta]$ be the subring defined by:
$$R_{\sigma} = \O[\chi^u \varpi^k \mid (u, k) \in \sigma^\vee \cap \tilde{M}],$$
and let $\U_{\sigma} = \Spec(R_{\sigma})$. 
When $\sigma$ is the cone over a polyhedron $\Delta$ in $N_\R \times \{1\}$ we also denote $R_{\sigma}$ by $R_\Delta$. By gluing the $\U_{\sigma}$, $\sigma \in \Sigma$, one constructs a toric scheme $\X_{\Sigma}$. When the fan $\Sigma$ is obtained from a polyhedral complex $\Sigma_1$ in $N_\R \times \{1\}$ we also denote the toric scheme $\X_{\Sigma}$ by $\X_{\Sigma_1}$. 
 
The next theorem is well-known and gives a classification of toric schemes (see \cite[Chap IV, Sec. 3, p. 191]{Kempf} as well as \cite[Section 3.5]{BPS}).
\begin{theorem}
One has the following:
\begin{itemize}
\item[(a)] The map $\sigma \mapsto \U_{\sigma}$ gives rise to an equivalence of categories between strictly convex rational polyhedral cones in $N_\R \times \R_{\geq 0}$ and affine toric $T$-schemes over $\Spec(\O)$.
\item[(b)] The map $\Sigma \mapsto \X_{\Sigma}$ gives an equivalence of categories between fans in $N_\R \times \R_{\geq 0}$ and toric $T$-schemes over $\Spec(\O)$.
\item[(c)] There is a bijection between the orbits of $T_\eta$ in $\X_\eta$ and the cones in $\Sigma_0$, that is the cones in $\Sigma$ that lie in $N_\R \times \{0\}$. Moreover, there is bijection between the orbits of $T_s$ in $\X_s$ and the cones of $\Sigma$ that intersect $N_\R \times \{1\}$, or in other words, the polyhedra in the polyhedral complex $\Sigma_1$. 
\item[(d)] In particular, the irreducible components of the special fiber $\X_s$ are in bijection with the vertices of $\Sigma$. Let 
$\rho$ be a ray in $\Sigma$ and let $(\nu_\rho, \ell_\rho)$ be the primitive vector on the ray generated by $\rho$. Then $\ell_\rho$ is the order of vanishing of the uniformizer $\varpi$ along the irreducible component of $\X_s$ corresponding to $\rho$.
\end{itemize}
\end{theorem}

\begin{example}  \label{ex-toric-scheme-P1}
Let $n=1$ and $K = \C((t))$ be the field of formal Laurent series in one indeterminate $t$.
Consider the polyhedral complex $\Sigma_1$ in $N_\R \times \{1\}$ consisting of the 
polyhedra $(-\infty, 0]$, $[0, 1]$, $[1, \infty)$ (Figure \ref{fig-toric-scheme-poly-P1}). The corresponding fan $\Sigma$ is depicted in Figure \ref{fig-toric-scheme-fan-P1}. Intuitively, the toric scheme $\X_{\Sigma}$ is the family over $\Spec(\O)$ whose generic fiber is a projective line $\mathbb{P}^1$ and the special fiber is a union of two copies of $\mathbb{P}^1$ glued at a fixed point (Figure \ref{fig-toric-scheme-P1}).
\begin{figure}[h]
    \centering
    \begin{tikzpicture}[>=stealth]
	\begin{axis}[xmin=-2.5, xmax=2.5, ymin=-0.5, ymax=2.3,
	axis lines*=middle, 
	xlabel=\empty,
	ylabel=\empty,
	xtick=\empty,
	xticklabels=\empty,
	ytick=\empty,
	yticklabels=\empty,]
	\addplot[name path=F,domain={-3:3}] {1.5};
	\plot[name path=I,domain={-3:3}, opacity=0] {2};
	\addplot[name path=G,domain={-0.5:2.3}] {\x};
	\addplot[name path=H,domain={-2:2}] {0} node [pos=0.97, below] {$\Sigma_0$};
	\addplot[pattern=north east lines, pattern color=blue]fill between[of=I and H, soft clip={domain=-2:0}];
	\addplot[pattern=horizontal lines, pattern color=red]fill between[of=I and G, soft clip={domain=0:2}];
	\addplot[pattern=vertical lines, pattern color=green]fill between[of=H and G, soft clip={domain=0:2}];
	\node at (axis cs:0,0) [circle, scale=0.4, draw=black!80,fill=black!80] {};
	\node at (axis cs:0,1.5) [circle, scale=0.4, draw=black!80,fill=black!80] {};
	\node at (axis cs:1.5,1.5) [circle, scale=0.4, draw=black!80,fill=black!80] {};
	\node at (axis cs:2.3,2) [] {$\widetilde{\Sigma}$};
	\end{axis}
	\end{tikzpicture}
        
    \caption{A fan $\Sigma$ defining a toric scheme with $\mathbb{P}^1$ as general fiber}
    \label{fig-toric-scheme-fan-P1}
\begin{tikzpicture}
    	\draw[-] (-3,0)--(3.2,0);
    	\draw (-3,1.5)--(3.5,1.5);
    	\draw[decorate,decoration={coil,segment length=4pt}] (-3,1.5)--(-0.2,1.5);
    	\draw[decorate,decoration={coil,segment length=4pt}] (0.2,1.5)--(1.3,1.5);
    	\draw[decorate,decoration={coil,segment length=4pt}] (1.7,1.5)--(3.5,1.5);
    	\draw[-] (0,-0.5)--(0,2.5);
	\node at  (0.4,2.5){$\R_{\geq 0}$};
	\node at  (0,1.5){$\bullet$};
	\node at  (1.5,1.5){$\bullet$};
	\node at  (3.5,0){$N_{\R}$};

		\end{tikzpicture}
\label{fig-toric-scheme-poly-P1}
\caption{The polyhedral complex $\Sigma$ associated to the fan $\Sigma$}

\vspace{.2cm}

\begin{tikzpicture}
\draw[thick](-1.7,1.3) circle (0.4);
\draw[thick] (-1.65,0.9) arc (-23:23:1);
\draw[thick, dashed] (-1.7,1.7) arc (157:203:1);
\draw[thick] (0.95,0.9) arc (-23:23:1);
\draw[thick, dashed] (0.95,1.7) arc (157:203:1);
\draw[thick] (1.75,0.9) arc (-23:23:1);
\draw[thick, dashed] (1.75,1.7) arc (157:203:1);
\draw[thick](0.9,1.3) circle (0.4);
\draw[thick](1.7,1.3) circle (0.4);
    \draw[-] (-2.5,0)--(2.5,0);
	\draw[dashed, ->] (-0.8,1.3)--(0.1,1.3);
    \draw[densely dotted, -] (-1.7,0.7)--(-1.7,0.2);
    \draw[densely dotted, -] (1.3,0.7)--(1.3,0.2);
    \node at  (-2.1,1.3){\tiny{$\bullet$}};
    \node at  (-1.3,1.3){\tiny{$\bullet$}};
    \node at  (0.5,1.3){\tiny{$\bullet$}};
    \node at  (1.3,1.3){\tiny{$\bullet$}};
    \node at  (2.1,1.3){\tiny{$\bullet$}};
	\node at  (-1.7,0){\tiny{$\bullet$}};
	\node at  (1.3,0){\tiny{$\bullet$}};
	\node at  (-1.7,-0.3){$\eta$};
	\node at  (1.3,-0.3){$s$};
	\node at  (-1.7,2.1){$\X_\eta$};
	\node at  (1.3,2.1){$\X_s$};
\end{tikzpicture}
\label{fig-toric-scheme-P1}
\caption{The general and special fibers of the toric scheme associated to the fan $\Sigma$}
\end{figure}
\end{example}

\begin{definition}   \label{def-tlb-over-S}
A \emph{toric line bundle} on $\X$ is a line bundle $\L$ over $\X$ together with a linear action of $T$ on $\L$.
We also define a \emph{framed} toric line bundle on $\X$ to be a toric line bundle together with a choice of a nonzero vector $e$ in the fiber $(\L_\eta)_{x_0}$. The choice of $e$ is equivalent to the choice of a $K$-linear isomorphism between $(\L_\eta)_{x_0}$ and $K$. 
\end{definition}

\begin{remark}    \label{rem-BPS}
Our notion of a framed toric line bundle is the same as the notion of a \emph{toric line bundle with a toric section} in \cite[Definition 3.6.4]{BPS}. We point out that we use the term \emph{toric line bundle} to mean a torus equivariant line bundle.
\end{remark}


We now recall the classification of toric line bundles on a toric scheme $\X_\Sigma$ in terms of piecewise linear and piecewise affine functions. This is similar to the classification of toric line bundles on a toric variety over a field. 
The following is essentially the same as \cite[Chap IV, Section 3, item (h)]{Kempf}.
\begin{theorem} \label{th-tlb-over-S}
There is a natural bijection between the set of framed toric line bundles $\mathcal{L}$ over $X_\Sigma$ and integral piecewise linear functions $\phi: \Sigma \to \R$. Moreover, the toric line bundle $\mathcal{L}_\eta$ over the generic fiber $\X_\eta$ corresponds to the piecewise linear function $\phi_0: |\Sigma_0| \to \R$, the restriction of $\phi$ to the fan $\Sigma_0$ (consisting of the cones in $\Sigma$ that lie in $N_\R \times \{0\}$). 
\end{theorem}

When $\Sigma$ is a complete fan in $N_\R \times \R_{\geq 0}$, the above can be stated in terms of piecewise affine functions on the polyhedral complex $\Sigma_1$.
Recall that a function $\phi_1: N_\R \times \{1\} \to \R$ is an \emph{affine function} if $\phi_1$ is a linear function   followed by a translation. That is, there is a linear function $\phi_0: N_\R \times \{0\} \to \R$ and $b \in \R$ such that $\phi_1(x, 1) = \phi_0(x, 0) + b$, $\forall x \in N_\R$. The linear function $\phi_0$ is called the \emph{linear part of $\phi_1$}. In another word, a function $\phi_1: N_\R \times \{1\} \to \R$ is affine if and only if it is the restriction of a \emph{linear function} $\phi: N_\R \times \R \to \R$ to $N_\R \times \{1\}$. We say that an affine function $\phi_1$ is \emph{integral} if $\phi_1(N \times \{1\}) \subset \Z$.

Let $\Sigma_1$ be a complete polyhedral complex in $N_\R \times \{1\}$. Recall that $|\Sigma_1|$ denotes the \emph{support} of $\Sigma_1$, the union of all the polyhedra in $\Sigma_1$. We say that a function $\phi_1:|\Sigma_1| \to \R$ is \emph{piecewise affine}, if for every polyhedron $\Delta \in \Sigma_1$, the function $\phi_{|\Delta}$ is the restriction of an affine function on $N_\R \times \{1\}$ to $\Delta$. We say that $\phi$ is an \emph{integral piecewise affine} function, if the maps $\phi_{|\Delta}$ are the restrictions of integral affine functions. The linear parts of the $\phi_{|\Delta}$, for all $\Delta \in \Sigma_1$, glue together to give a piecewise linear function $\phi_0: |\Sigma_1| \to \R$ which we call the \emph{piecewise linear part of} $\phi$. 

For a function $\phi: N_\R \times \R_{\geq 0} \to \R$ and $m \geq 0$ let $\phi_m: N_\R \times \{m\} \to \R$ denote the restriction of $\phi$ to $N_\R \times \{m\}$. Let $\Sigma$ be a complete fan in $N_\R \times \R_{\geq 0}$ with $\Sigma_1$ corresponding complete polyhedral complex in $N_\R \times \{1\}$. One verifies that $\phi \mapsto \phi_1$ gives a bijection between the collection of (integral) piecewise linear functions $\phi: |\Sigma| \to \R$ and (integral) piecewise affine functions $\phi_1: |\Sigma_1| \to \R$. Theorem \ref{th-tlb-over-S} can then be restated as follows:
\begin{corollary} \label{cor-tlb-over-S}
Let $\Sigma \subset N_\R \times \R_{\geq 0}$ be a complete fan. With notation as above, there is a natural bijection between the set of framed toric line bundles $\mathcal{L}$ over $X_\Sigma$ and integral piecewise affine functions $\phi_1: |\Sigma_1| \to \R$. Moreover, the toric line bundle $\mathcal{L}_\eta$ over the generic fiber $\X_\eta$ corresponds to $\phi_0: |\Sigma_0| \to \R$, the linear part of $\phi_1$. 
\end{corollary}

Our Theorems \ref{th-main} and \ref{th-main2} extend the above (Theorem \ref{th-tlb-over-S} and Corollary \ref{cor-tlb-over-S}) to toric vector bundles over $\Spec(\O)$.

\section{Preliminaries on additive norms and valuations}  \label{sec-add-norms}
In this section we review some basic facts about valuations and additive norms on a finite dimensional $K$-vector space $E$. These are related to the geometric realizations of the Tits and Bruhat-Tits buildings of the reductive algebraic group $\GL(E)$. These are the main gadgets in our main theorems on classifying toric vector bundles over toric schemes (Section \ref{sec-tvbs-dvr}). 

As usual let $E$ be an $r$-dimensional vector space over a field $K$ equipped with a discrete valuation $\val: K \to \overline{\Z}$. We will be interested in two kinds of ``valuations'' on $E$, first the ones that extend the valuation $\val$ on $K$ to $E$, and secondly the ones that extend the trivial valuation on $K$ to $E$. Following terminology from Bruhat-Tits theory for $\GL(E)$, we refer to the first kind of valuations as \emph{additive norms}. We will refer to the second kind simply as \emph{valautions} on $E$. More precisely we make the following definitions:

\begin{definition}[Additive norm] \label{def-add-norm}
We call a function $v: E \to \overline{\R} = \R \cup \{\infty\}$ an {\it additive norm} if the following hold:
\begin{itemize}
	\item[(1)] For all $e \in E$ and $\lambda \in K$ we have $v(\lambda e) = \val(\lambda) + v(e)$. 
	\item[(2)] For all $e_1, e_2 \in E$, the non-Archimedean inequality $v(e_1+e_2) \geq \min\{v(e_1), v(e_2)\}$ holds.
	\item[(3)] $v(e) = \infty$ if and only if $e=0$. 
\end{itemize}
An additive norm is \emph{integral} if it attains values in $\overline{\Z}$. 
\end{definition}

Note that if $v$ is an additive norm then for any $a \in \R$, $v_a$ defined by $v_a(e) = v(e) + a$ is also an additive norm. Two additive norms $v$, $v'$ are said to be \emph{equivalent} if their difference is a constant.

\begin{definition}[Valuation]  \label{def-val}
We call a function $v: E \to \overline{\R}$ a {\it valuation} if the following hold:
\begin{itemize}
	\item[(1)] For all $e \in E$ and $0 \neq \lambda \in K$ we have $v(\lambda e) = v(e)$. 
	\item[(2)] For all $e_1, e_2 \in E$, the non-Archimedean inequality $v(e_1+e_2) \geq \min\{v(e_1), v(e_2)\}$ holds.
	\item[(3)] $v(e) = \infty$ if and only if $e=0$. 
\end{itemize}
A valuation is \emph{integral} if it attains values in $\overline{\Z} = \Z \cup \{\infty\}$. 
\end{definition}
More generally, we define the notion of a level $m$ additive norm which extends both additive norms and valuations. 
\begin{definition}[Level $m$ additive norm]  \label{def-level-m-add-norm}
For $m \geq 0$, we call a function $v: E \to \overline{\R} = \R \cup \{\infty\}$ a {\it level $m$ additive norm} if the following hold:
\begin{itemize}
	\item[(1)] For all $e \in E$ and $0 \neq \lambda \in K$ we have $v(\lambda e) = m \val(\lambda) + v(e)$. 
	\item[(2)] For all $e_1, e_2 \in E$, the non-Archimedean inequality $v(e_1+e_2) \geq \min\{v(e_1), v(e_2)\}$ holds.
	\item[(3)] $v(e) = \infty$ if and only if $e=0$. 
\end{itemize}
Clearly, an additive norm is a level $1$ additive norm and a valuation is a level $0$ additive norm.
\end{definition}

\begin{definition}[Extended buildings]
Throughout we use the following notation and terminology:
\begin{itemize}
\item[(a)] We denote the set of all additive norms on $E$ by $\tilde{\B}_\aff(E)$ and call it the \emph{extended Bruhat-Tits building of $E$}. In the literature, also it is  sometimes called the \emph{Goldman-Iwahori space of $E$}. 
\item[(b)] We denote the set of all valuations on $E$ by $\tilde{\B}_\sph(E)$ and call it the \emph{extended Tits building of $E$} (or the \emph{cone over the Tits building of $E$}).
\item[(c)] For $m \geq 0$, we denote the set of all level $m$ additive norms by $\tilde{\B}_m(E)$. Clearly $\tilde{\B}_1(E) = \tilde{\B}_\aff(E)$ and $\tilde{\B}_0(E) = \tilde{\B}_\sph(E)$. We let $$\tilde{\B}(E) = \bigsqcup_{m \geq 0} \tilde{\B}_m(E),$$
and call it the \emph{total Bruhat-Tits building of $E$} (or the \emph{total building of $E$ for short}).
\end{itemize}
\end{definition}

\begin{remark}
We note that $\tilde{\B}(E)$ comes with a natural action of the multiplicative semigroup $\R_{\geq 0}$: if $v: E \to \overline{\R}$ is a level $m$ additive norm and $\lambda \in \R$ then one verifies that $\lambda v$ is a level $m \lambda$ additive norm. 
\end{remark}

A \emph{frame} is a direct sum decomposition of $E = \sum_{i=1}^r L_i$ into one-dimensional subspaces $L_i$. In other words, frames correspond to vector space bases up to multiplication of basis elements by nonzero scalars. 

We say that a level $m$ additive norm $v: E \to \overline{\R}$ is \emph{adapted} to a frame $L=\{L_1, \ldots, L_r\}$ if the following holds. For any $e = \sum_i e_i$, $e_i \in L_i$, we have: 
$$v(e) =  \min\{ v(e_i) \mid i=1, \ldots, r\}.$$
In other words, if $B = \{b_1, \ldots, b_r\}$ is a basis with $b_i \in L_i$, then for any $e = \sum_i \lambda_i b_i$ we have: $$v(e) = \min\{ m\val(\lambda_i) + v(b_i) \mid i=1, \ldots, r\}.$$ That is, the additive norm $v$ is determined by its values on the basis $B$. 

\begin{definition}[Extended apartment] \label{def-ext-apt}
Let $L = \{L_1, \ldots, L_r\}$ be a frame for $E$. We denote the collection of all level $m$ additive norms that are adapted to $L$ by $\tilde{A}_m(L)$ and call it the extended level $m$ apartment of $L$. We put $\tilde{A}(L) = \bigsqcup_{m \geq 0} \tilde{A}_m(L)$. We also set 
$\tilde{A}_\sph(L) = \tilde{A}_0(L)$ and $\tilde{A}_\aff(L) = \tilde{A}_1(L)$.
If $B$ is a basis for $E$ spanning a frame $L$, we also denote the extended apartments corresponding to $L$ by $\tilde{A}(B)$, $\tilde{A}_\sph(B)$ and $\tilde{A}_\aff(B)$, respectively.
\end{definition}

\subsubsection{Additive norms, lattices and affine Grassmannian} \label{subsubsec-lattices}
An \emph{$\O$-lattice} (or \emph{lattice} for short) $\Lambda$ in $E$ is a rank $r$ free $\O$-submodule in $E$. 
One sees that there is a one-to-one correspondence between the $\O$-lattices in $E$ and the integral additive norms on $E$ as follows: To an integral additive norm $v$ one assigns the lattice $\Lambda_v$ given by:
\begin{equation}   \label{equ-Lambda-v}
\Lambda_v = \{ e \in E \mid v(e) \geq 0\}.
\end{equation}
Conversely, a lattice $\Lambda$ gives rise to an additive norm $v_\Lambda: E \to \overline{\Z}$ by:
$$v_\Lambda(e) = \max\{ k \in \Z \mid \varpi^{-k} e \in \Lambda \}.$$
One verifies that the maps
$v \mapsto \Lambda_v$ and $\Lambda \mapsto v_\Lambda$ are inverses of each other and hence provide a one-to-one correspondence between integral additive norms and lattices.  

Fix a basis for $E$ to identify it with $K^r$. The group $\GL(r, K)$ acts transitively on the set of all $\O$-lattices and the stabilizer of the standard lattice $\O^r$ is $\GL(r, \O)$. Thus, the collection of all $\O$-lattices in $K^r$ can be identified with $\GL(r, K) / \GL(r, \O)$. It is usually referred to as the \emph{affine Grassmannian} of the group $\GL(r, K)$. 

\subsubsection{Valuations and labeled flags} \label{subsubsec-flags}
A labeled flag $(F_\bullet, c_\bullet)$ is a flag of subspaces in $E$:
$$ F_\bullet = (\{0\} = F_0 \subsetneqq F_1 \subsetneqq \cdots \subsetneqq F_k = E),$$
and a sequence of decreasing numbers $c_\bullet = (c_1 > \cdots > c_k)$. 
Each valuation $v: E \to \overline{R}$ naturally gives rise to a labeled flag as follows: let $v(E) = \{\infty=c_0 > c_1 > \cdots > c_k\}$ be the value set of $v$. For each $i=1, \ldots, k$, put $$F_i = E_{\geq c_i} = \{e \in E \mid v(e) \geq c_i\}.$$
Conversely, given a labeled flag $(F_\bullet, c_\bullet)$ we define a valuation $v$ by $v(e) = \max\{ c_i \mid e \in F_i\}$, $\forall e \in E$. One verifies that the above gives a one-to-one correspondence between valuations and labeled flags. 

\subsubsection{Floor, ceiling and pull-back of an additive norm}  \label{subsubsec-ceiling}
Next we discuss some basic operations on additive norms, namely the floor and ceiling functions and pull-back, and prove some facts about them for future use in Section \ref{subsec-main}. 

One has a natural partial order $\preceq$ on the set $\tilde{\B}(E)$ of all additive norms. 
Let $v$, $v'$ be additive norms on $E$. We say that $v \preceq v'$ if $v(e) \leq v'(e)$, for all $e \in E$.

Recall that for $r \in \R$, the floor $\lfloor r \rfloor$ (respectively the ceiling $\lceil r \rceil$) is the largest integer less than or equal to $r$ (respectively the smallest integer greater than or equal to $r$).
\begin{definition}[Floor and ceiling of an additive norm]   \label{def-ceiling}
For an additive norm $v: E \to \overline{\R}$ we define its \emph{floor} $\lfloor v \rfloor$ and  \emph{ceiling} $\lceil v \rceil: E \to \overline{\Z}$ to be functions from $E$ to $\overline{\Z}$ given by:
$$\lfloor v \rfloor(e) = \lfloor v(e) \rfloor,$$
$$\lceil v \rceil(e) = \lceil v(e) \rceil, \quad \forall e \in E.$$
\end{definition}

\begin{proposition}
$\lfloor v \rfloor$ and $\lceil v \rceil$ are additive norms.	
\end{proposition}  

\begin{proof}
We write the proof for the ceiling, the proof for the floor is similar. We know that there exists a basis $B = \{b_1, \ldots, b_r\}$ such that $v$ is adapted to $B$. Thus:
$$v\left(\sum_i \lambda_i b_i \right) = \min\{ \val(\lambda_i)+v(b_i) \mid i=1, \ldots, r\}.$$
Now
\begin{align*}
\lceil v \rceil \left(\sum_i \lambda_i b_i\right) &= \lceil v(\sum_i \lambda_i b_i) \rceil \\
&= \lceil \min\{ \val(\lambda_i)+v(b_i) \mid i=1, \ldots, r\} \rceil \\ &= \min\{ \val(\lambda_i)+\lceil v(b_i) \rceil \mid i=1, \ldots, r\}.\\
\end{align*}
But any expression $\min\{ \val(\lambda_i)+a_i \mid i=1, \ldots, r\}$, $a_i \in \R$, defines an additive norm. This proves that $\lceil v \rceil$ is indeed an additive norm. 
\end{proof}

The floor (respectively ceiling) of an additive norm $v$ is the greatest integral additive norm that is smaller than or equal to $v$ (respectively smallest integral additive norm that is greater than or equal to $v$). 

\begin{definition}[Pull-back of an additive norm] \label{def-pull-back}
Let $F: E \to E'$ be a linear map between $K$-vector spaces $E$ and $E'$. Given an additive norm $v': E \to \overline{\R}$ one can consider its pull-back $F^*(v')$ given by $$F^*(v')(e) = v'(F(e)), \quad \forall e \in E.$$
\end{definition}
We note that $F^*(v')$ may not be an additive norm on $E$ for the following reason: $F$ may have nonzero kernel in which case $F^*(v')$ assigns value $\infty$ to nonzero elements in the kernel. We will need the following in Section \ref{subsec-main}.

\begin{proposition}  \label{prop-F-v-v'}
Let $F: E \to E'$ be a linear map between $K$-vector spaces $E$ and $E'$. Let $v$ and $v'$ be integral additive norms on $E$ and $E'$ respectively. Then, for all $e \in E$,  $$v(e) \leq F^*(v')(e)$$ if and only if $$F(\Lambda_v) \subset \Lambda_{v'}.$$ Recall that $\Lambda_v$, $\Lambda_{v'}$ denote the lattices corresponding to the integral additive norms $v$, $v'$ respectively (see \eqref{equ-Lambda-v}).   \end{proposition}
\begin{proof}
Suppose $v(e) \leq F^*(v')(e)$, for all $e \in E$, but $F(\Lambda_v) \not\subset \Lambda_{v'}$. Then there exists $e \in \Lambda_v$ with $F(e) \notin \Lambda_{v'}$. Since $e \in \Lambda_v$ we have $v(e) \geq 0$ and since $F(e) \notin \Lambda_{v'}$ we have $v'(F(e)) < 0$, the contradiction proves the claim. Conversely, suppose $F(\Lambda_v) \subset \Lambda_{v'}$ then for any $e \in E$ we have:
\begin{align*}
v(e) &= \max\{ k \mid \varpi^{-k} e \in \Lambda_v\}, \\
& \leq \max\{ k \mid \varpi^{-k} F(e) \in F(\Lambda_v) \}, \\
&\leq \max\{ k \mid \varpi^{-k} F(e) \in \Lambda_{v'} \},\\
&= v'(F(e)).\\
\end{align*}
\end{proof}


\section{Tits and Bruhat-Tits buildings of $\GL(E)$}
\label{sec-building}

A \emph{building} is an (abstract) simplicial complex together with a collection of distinguished subcomplexes called \emph{apartments} that satisfy certain axioms (see \cite[Definition 4.1]{AB}). One can think of the notion of building as a combinatorial version of the notion of symmetric space from Lie theory and differential geometry. 

There are two important kinds of buildings: \emph{spherical buildings} and \emph{affine buildings}. 
\begin{itemize}
\item In a spherical building, every apartment is a triangulation of a sphere. To an algebraic group $G$ over a field there corresponds a spherical building which is usually referred to as 
the \emph{Tits building} of $G$. As a simplicial complex, it encodes the relative position of \emph{parabolic subgroups} of $G$.
\item In an affine building every apartment is a triangulation of an affine (Euclidean) space. To an algebraic group $G$ over a discretely valued field there corresponds an affine building which is usually referred to as the \emph{Bruhat-Tits building} of $G$. As a simplicial complex, it encodes the relative position of \emph{parahoric subgroups} of $G$.
\end{itemize}

As mentioned before, in this paper, we are only concerned with the Bruhat-Tits building of a general linear group $\GL(r, K)$ and we do not discuss the technical construction of the Bruhat-Tits building of an arbitrary linear algebraic group. For the general construction we refer the reader to \cite{Bruhat-Tits1} or or \cite[Section 3]{RTW}. 
For the general theory of buildings we refer the interested reader to the monographs \cite{AB} and \cite{Garrett}.

\subsection{The Bruhat-Tits building of $\GL(r, K)$}     \label{subsec-BT-bldg}
In this section we review the descriptions of the Bruhat-Tits building of the general group in terms of lattices and additive norms. The reader can also consult \cite[Section 6.9]{AB} and \cite[Section 1.2.1]{RTW}.


Let $E \cong K^r$ be an $r$-dimensional vector space over a discretely valued field $K$. We first describe the Bruhat-Tits building of $\GL(E)$ as an abstract simplicial complex which we denote by $\Delta_\aff(E)$. This description is in terms of of \emph{lattices} in $E$. Then we describe the geometric realization of the simplicial complex $\Delta_\aff(E)$ as a topological space which we denote by $\B_\aff(E)$. This geometric realization is in terms of \emph{additive norms} on $E$. Every simplex in $\Delta_\aff(E)$ corresponds to a subset of $\B_\aff(E)$ homeomorphic to a standard simplex. By abuse of terminology we refer to both $\Delta_\aff(E)$ and $\B_\aff(E)$ as the Bruhat-Tits building of $\GL(E)$.  

Recall that a lattice $\Lambda$ in $E$ is a rank $r$ free $\O$-submodule in $E$, and two lattices $\Lambda_1$, $\Lambda_2$ are said to be equivalent if for some nonzero $c \in K^\times$ we have $c\Lambda_1 = \Lambda_2$. We denote the equivalence class of a lattice $\Lambda$ by $[\Lambda]$.

The vertices in the Bruhat-Tits building $\Delta_\aff(E)$ correspond to the equivalence classes of lattices in $E$. Two vertices $[\Lambda_1]$ and $[\Lambda_2]$ are connected if for some representatives $\Lambda_1$, $\Lambda_2$ we have $\varpi \Lambda_1 \subsetneqq \Lambda_2 \subsetneqq \Lambda_1$. A $(k-1)$-simplex in $\Delta_\aff(E)$ corresponds to $k$ vertices $[\Lambda_1], \ldots, [\Lambda_{k}]$ such that:
\begin{equation}  \label{equ-periodic-lattice-chain}
	\varpi \Lambda_k \subsetneqq \Lambda_1 \subsetneqq \cdots \subsetneqq \Lambda_k.
\end{equation}
We call a chain of lattices as in \eqref{equ-periodic-lattice-chain} a \emph{periodic lattice chain}.

A \emph{full periodic lattice} chain is a chain of lattices 

\begin{equation}
	\varpi\Lambda_r\subsetneqq \Lambda_1\subsetneqq \hdots \subsetneqq \Lambda_r.
\end{equation}
Then each subsequent quotient $\Lambda_{i+1}/\Lambda_{i}$ is a locally free $\mathcal{O}$-module of rank $1$. The full periodic lattice chains correspond to simplices of maximum dimension (also called \emph{chambers}).

Let $\{b_1, \ldots, b_r\}$ be a $K$-vector space basis for $E$ and let $\Lambda_i\subset E$ be the free 
$\mathcal{O}$-module (lattice) generated by 
$\{b_1,\ldots, b_i,\varpi b_{i+1}, \ldots,\varpi b_r\}$. The chain of lattices 
\begin{equation}   \label{equ-full-lattice-chain}
\bigoplus_{i=1}^r \varpi\,\O\, b_i = \varpi \Lambda_r \subsetneqq\Lambda_{1}\subsetneqq \hdots \subsetneqq \Lambda_r= \bigoplus_{i=1}^r \O\, b_i
\end{equation}
is a full periodic lattice chain. Moreover, every full periodic lattice chain has the above form \eqref{equ-full-lattice-chain} for some choice of a basis $B$. 

Also recall that a frame $L=\{L_1, \ldots, L_r\}$ is a direct sum decomposition $E = \bigoplus_{i=1}^r L_i$ of $E$ into $1$-dimensional subspaces $L_i$. 
To a frame $L$ there corresponds a subcomplex of $\Delta_\aff(E)$ called the \emph{apartment} associated to $L$ as follows. Let $B=\{b_1, \ldots, b_r\}$ be a basis of $E$ with $b_i \in L_i$.
The vertices in the apartment are equivalence classes $[\Lambda]$ where $\Lambda$ is of the form
\begin{equation} \label{equ-affine-apt}
	\Lambda = \bigoplus_{i=1}^r \O \varpi^{a_i} b_i, 
\end{equation}
for any choice of $a_i \in \Z$. Note that a different choice of a basis $B$ in the same frame $L$ gives rise to the same collection of vertices but possibly with different coordinates $(a_1, \ldots, a_r)$. More precisely, let $\{b'_1, \ldots, b'_r\}$ be another basis with $b'_i = \lambda_i b_i$, $0 \neq \lambda_i \in K$. Then $\Lambda$ from \eqref{equ-affine-apt} can be written as $$\Lambda = \bigoplus_{i=1}^r \O \varpi^{a_i-\val(\lambda_i)} b'_i.$$
Clearly, if $(a_1, \ldots, a_r)$ runs over all the points in $\Z^r$ then $(a_1-\val(\lambda_1), \ldots, a_r-\val(\lambda_r))$ also runs over all the points in $\Z^r$. 
The set of vertices in the apartment is a copy of $\Z^r$ but without a specified point as the origin. A choice of a basis $B$, spanning the frame $L$, distinguishes a point in the apartment as the origin. 

\begin{remark}[Iwahori and parahoric subgroups] \label{rem-Iwahori-subgp} 
The $\GL(E)$-stabilizer of a full periodic lattice chain  is called an \emph{Iwahori subgroup}. Let us fix a basis $B$ and identify $E$ with $K^r$. Then the Iwahori subgroup that is the stabilizer of a full lattice chain as in \eqref{equ-full-lattice-chain} is equal to the preimage of a Borel subgroup $B \subset \GL(r, \k)$ under the evaluation map $p: \GL(r, \mathcal{O})\rightarrow \GL(r, \k)$ sending the unifromizer $\varpi$ to $0$. The stabilizer of a periodic lattice chain is called a \emph{parahoric subgroup}. It is the preimage, under $p$, of a parabolic subgroup in $\GL(r, \Bbbk)$.
Each frame $L$ in $E$ determines a split maximal torus $H_L \subset \GL(E)$. The simplices in the apartment of $L$ correspond to the parahoric subgroups in $\GL(E)$ that contain the maximal torus $H_L$. 
\end{remark}

Next we talk about the geometric realization of the abstract simplicial complex $\Delta_\aff(E)$. By the \emph{geometric realization of $\Delta_\aff(E)$} we mean a topological space $\B(E)$ together with a triangulation (that is, a decomposition into subsets homeomorphic to standard simplices) such that the simplices in the triangulation give the same abstract simplicial complex as $\Delta_\aff(E)$.

We say that two additive norms $v_1$, $v_2$ on $E$ are \emph{equivalent} if their difference is a constant. 

\begin{definition}[Geometric realization of Bruhat-Tits building]
We define $\B_\aff(E)$, the \emph{(geometric realization of) the Bruhat-Tits building} of $\GL(E)$ to be the set of equivalence classes of additive norms on $E$. For a frame $L$, the \emph{(geometric realization of) the apartment} associated to $L$, denoted by $A_\aff(L)$, is the set of equivalence classes of all additive norms adapted to $L$.
Let $B = \{b_1, \ldots, b_r\}$ be a basis spanning the frame $L$. Each additive norm $v$ adapted to $L$ is uniquely determined by the values $\{v(b_1), \ldots, v(b_r)\}$. It follows that $A_\aff(L)$ can be identified with the $(r-1)$-dimensional affine space $\R^r / \R \cong \R^{r-1}$ (since we are taking equivalence classes of additive norms). 
\end{definition}

The space $\B_\aff(E)$ has a natural topology and a triangulation. For a frame $E$, a choice of a basis $B$ spanning the frame $L$ gives an identification of the apartment $A_\aff(L)$ with the affine space $\R^{r-1}$.

\begin{figure}
	\begin{tikzpicture}
    	\draw[-] (-3,0)--(2,0);
	\node at  (-2,0){$\bullet$};
	\node at  (-1,0){$\bullet$};
	\node at  (0,0){$\bullet$};
	\node at  (1,0){$\bullet$};

		\end{tikzpicture}
		\vspace{.5cm}
	\caption{An apartment in the Bruhat-Tits/affine building of $\GL(2)$, a subdivision of the affine line into identical segments. An extended apartment is the Cartesian product of this picture with a copy of $\R$}
	
	\vspace{.5cm}
	\begin{tikzpicture}[scale=0.5]
	\begin{axis}[xmin=-0.9, xmax=0.9, ymin=-0.7, ymax=0.7,
	axis lines*=middle, axis line style={draw=none},
	xlabel=\empty,
	ylabel=\empty,
	xtick=\empty,
	xticklabels=\empty,
	ytick=\empty,
	yticklabels=\empty]
	\addplot[domain={-1.2:1.2}] {2*\x};
	\addplot[domain={-2.2:0.2}] {2*\x+1};
	\addplot[domain={-0.2:2.2}] {2*\x-1};
	\addplot[domain={-2.2:0.2}] {2*\x+2};
	\addplot[domain={-0.2:3.2}] {2*\x-2};
	\addplot[domain={-1.2:1.2}] {-2*\x};
	\addplot[domain={-2.2:0.2}] {-2*\x-1};
	\addplot[domain={-0.2:2.2}] {-2*\x+1};
	\addplot[domain={-2.2:0.2}] {-2*\x-2};
	\addplot[domain={-0.2:3.2}] {-2*\x+2};
	\addplot[domain={-3:3}] {0.5};
	\addplot[domain={-3:3}] {0};
	\addplot[domain={-3:3}] {-0.5};
	\node at (axis cs:0,0) [circle, scale=0.5, draw=black!80,fill=black!80] {};
	\node at (axis cs:0.5,0) [circle, scale=0.5, draw=black!80,fill=black!80] {};
	\node at (axis cs:-0.5,0) [circle, scale=0.5, draw=black!80,fill=black!80] {};
	\node at (axis cs:0.25,0.5) [circle, scale=0.5, draw=black!80,fill=black!80] {};
	\node at (axis cs:-0.25,0.5) [circle, scale=0.5, draw=black!80,fill=black!80] {};
	\node at (axis cs:0.75,0.5) [circle, scale=0.5, draw=black!80,fill=black!80] {};
	\node at (axis cs:-0.75,0.5) [circle, scale=0.5, draw=black!80,fill=black!80] {};
	\node at (axis cs:0.25,-0.5) [circle, scale=0.5, draw=black!80,fill=black!80] {};
	\node at (axis cs:-0.25,-0.5) [circle, scale=0.5, draw=black!80,fill=black!80] {};
	\node at (axis cs:0.75,-0.5) [circle, scale=0.5, draw=black!80,fill=black!80] {};
	\node at (axis cs:-0.75,-0.5) [circle, scale=0.5, draw=black!80,fill=black!80] {};
	\end{axis}
	\end{tikzpicture}

	\caption{An apartment in the Bruhat-Tits/affine building of $\GL(3)$, a triangulation of the affine plane. An extended apartment is the Cartesian product of this picture with a copy of $\R$}
\end{figure}


\subsubsection{Bruhat-Tits building of $\GL(2, K)$} \label{subsubsec-GL2}
Let $B = \{b_1, b_2\}$ be a basis for $E = K^2$. Consider the lattice $$\Lambda = \O b_1 \oplus \O b_2.$$ The equivalence class $[\Lambda]$ is comprised of the lattices $\O \varpi^{a} b_1 \oplus \O \varpi^{a} b_2$, $a \in \Z$. 
The vertices in the apartment $A_\aff(B)$ consist of all the equivalence classes of lattices of the form:
$$\Lambda_{a_1, a_2} = \O \varpi^{a_1} b_1 \oplus \O \varpi^{a_2} b_2.$$ Thus the set of vertices in $A_\aff(B)$ is $\{[\Lambda_{a, 0}] \mid a \in \Z\}$ and $A_\aff(B)$ is a $1$-dimensional affine space. 
	
A vertex $[\Lambda']$ is connected to the vertex $[\Lambda]$ if for some representative $\Lambda'$ we have:
	$$\varpi \Lambda \subsetneqq \Lambda' \subsetneqq \Lambda.$$
	Let $\overline{\Lambda'}$ be the image of $\Lambda'$ in the $\k$-vector space $\Lambda/\varpi \Lambda \cong \k^2$. 
One sees that $\Lambda' \mapsto \overline{\Lambda'}$ gives a one-to-one correspondence between the lattices $\Lambda'$, $\varpi \Lambda \subsetneqq \Lambda' \subsetneqq \Lambda$, and the $\k$-vector subspaces $L$, $\{0\} \subsetneqq L \subsetneqq \Lambda/\varpi \Lambda$. It follows that the set of vertices $[\Lambda']$ connected to $[\Lambda]$ is in one-to-one correspondence with the points on the projective line $\mathbb{P}(\Lambda/\varpi \Lambda) \cong \mathbb{P}^1$. In particular, if the residue field $\k$ is a finite field with $q$ elements, then every vertex is connected to $q+1$ other vertices. It can also be shown that the graph obtained by connecting the vertices as above is a tree, that is, there are no cycles. Moreover, if the base field $K$ is complete then 
	any one-sided infinite path lies in an apartment
	(see \cite[Chap. II, \S1.1]{Serre} and \cite[Section I.2 ]{Casselman-SL2}). An example is $K = \Q_p$, the field of $p$-adic numbers (Figure \ref{fig-Tree-Q2-Q3}). 
	
	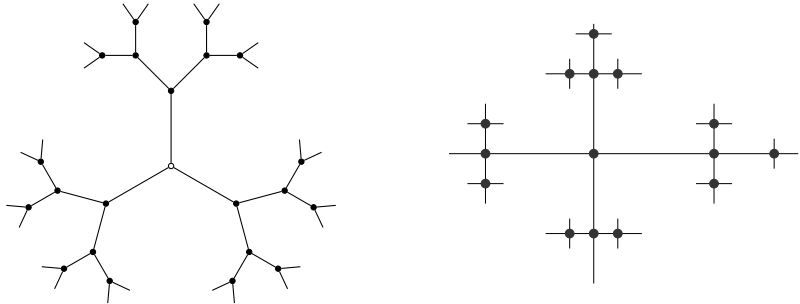
\begin{figure}
		\centering

\begin{tikzpicture}[
  grow cyclic,
  level distance=1cm,
  level/.style={
    level distance/.expanded=\ifnum#1>1 \tikzleveldistance/1.5\else\tikzleveldistance\fi,
    nodes/.expanded={\ifodd#1 fill\else fill\fi}
  },
  level 1/.style={sibling angle=120},
  level 2/.style={sibling angle=90},
  level 3/.style={sibling angle=90},
  level 4/.style={sibling angle=70},
  nodes={circle,draw,inner sep=+0pt, minimum size=2pt},
  ]
\path[rotate=-30]
  node {}
  child foreach \cntI in {1,...,3} {
    node {}
    child foreach \cntII in {1,...,2} { 
      node {}
      child foreach \cntIII in {1,...,2} {
        node {}
        child foreach \cntIV in {1,...,2} 
      }
    }
  };
\end{tikzpicture}
\hspace{0.5in}
	\begin{tikzpicture}[scale=0.7]
	\begin{axis}[xmin=-1.5, xmax=1.5, ymin=-1.5, ymax=1.5,
	axis lines*=middle, axis line style={draw=none},
	xlabel=\empty,
	ylabel=\empty,
	xtick=\empty,
	xticklabels=\empty,
	ytick=\empty,
	yticklabels=\empty]
	\addplot[domain={-1.6:1.4}] {0};
	\addplot[color=black] coordinates {(-0.3,-1.3) (-0.3,1.3)};
	\addplot[domain={-0.7:0.1}] {0.8};
	\addplot[domain={-0.7:0.1}] {-0.8};
	\addplot[color=black] coordinates {(-1.2,-0.5) (-1.2,0.5)};
	\addplot[color=black] coordinates {(0.7,-0.5) (0.7,0.5)};
	\addplot[domain={-1.35:-1.05}] {0.3};
	\addplot[domain={-1.35:-1.05}] {-0.3};
	\addplot[domain={0.55:0.85}] {0.3};
	\addplot[domain={0.55:0.85}] {-0.3};
	\addplot[color=black] coordinates {(-0.5,0.65) (-0.5,0.95)};
	\addplot[color=black] coordinates {(-0.1,0.65) (-0.1,0.95)};
	\addplot[color=black] coordinates {(-0.5,-0.65) (-0.5,-0.95)};
	\addplot[color=black] coordinates {(-0.1,-0.65) (-0.1,-0.95)};
	\addplot[color=black] coordinates {(1.2,-0.15) (1.2,0.15)};
	\addplot[domain={-0.45:-0.15}] {1.2};
	\node at (axis cs:-0.3,0) [circle, scale=0.5, draw=black!80,fill=black!80] {};
	\node at (axis cs:-0.3,-0.8) [circle, scale=0.5, draw=black!80,fill=black!80] {};
	\node at (axis cs:-0.3,0.8) [circle, scale=0.5, draw=black!80,fill=black!80] {};
	\node at (axis cs:-0.5,0.8) [circle, scale=0.5, draw=black!80,fill=black!80] {};
	\node at (axis cs:-0.1,0.8) [circle, scale=0.5, draw=black!80,fill=black!80] {};
	\node at (axis cs:-0.5,-0.8) [circle, scale=0.5, draw=black!80,fill=black!80] {};
	\node at (axis cs:-0.1,-0.8) [circle, scale=0.5, draw=black!80,fill=black!80] {};
	\node at (axis cs:-0.3,1.2) [circle, scale=0.5, draw=black!80,fill=black!80] {};
	\node at (axis cs:-1.2,0.3) [circle, scale=0.5, draw=black!80,fill=black!80] {};
	\node at (axis cs:-1.2,-0.3) [circle, scale=0.5, draw=black!80,fill=black!80] {};
	\node at (axis cs:0.7,0.3) [circle, scale=0.5, draw=black!80,fill=black!80] {};
	\node at (axis cs:0.7,-0.3) [circle, scale=0.5, draw=black!80,fill=black!80] {};
	\node at (axis cs:-1.2,0) [circle, scale=0.5, draw=black!80,fill=black!80] {};
	\node at (axis cs:0.7,0) [circle, scale=0.5, draw=black!80,fill=black!80] {};
	\node at (axis cs:1.2,0) [circle, scale=0.5, draw=black!80,fill=black!80] {};

\end{axis}
\end{tikzpicture}
		\caption{The Bruhat-Tits buildings of $\GL(2, \Q_2)$ and $\GL(2, \Q_3)$}
		\label{fig-Tree-Q2-Q3}
	\end{figure}
	
	Let us determine for what other bases $B' = \{b'_1, b'_2\}$ the vertex $[\Lambda]$ lies in the corresponding apartment $A_{B'}$. If $[\Lambda]$ lies in $A_{B'}$ then we can find $\gamma_1$, $\gamma_2 \in \Z$ such that:
	\begin{equation}   \label{equ-Lambda-B-B'}
		\O b_1 \oplus \O b_2 = \Lambda = \O \varpi^{\gamma_1} b'_1 \oplus \O \varpi^{\gamma_2} b'_2.
	\end{equation}
	Let $C = [c_{ij}]$, $1 \leq i, j \leq 2$, be the change of coordinates matrix from $B$ to $B'$. That is, $b'_i = \sum\limits_{j=1}^2 c_{ij} b_j$. The equation \eqref{equ-Lambda-B-B'} is then equivalent to 
	$$C D \in \GL(2, \O),$$ where $D = \operatorname{diag}(\varpi^{\gamma_1}, \varpi^{\gamma_2})$. 
	For instance, we can take $\gamma_1 = \gamma_2 = 0$ and $C = \begin{bmatrix} 1 &  \varpi^\delta \\ 0 & 1 \\ \end{bmatrix}$ with $\delta \geq 0$. Then $C D = C \in \GL(2, \O)$ and thus
	$$\O b_1 \oplus \O b_2 = \Lambda = \O b_1 \oplus \O (\varpi^\delta b_1 + b_2).$$ Hence $\Lambda$ belongs to the apartment $A_{B'}$ where $B' = \{b_1, \varpi^\delta b_1 + b_2 \}$. 

To demonstrate how apartments can overlap we consider the apartments associated to the bases $B = \{b_1, b_2\}$ and $B' = \{b_1+b_2, b_2\}$. The vertices in the apartments $A_B$ and $A_{B'}$ correspond respectively to the chain of lattices:  	
$$\cdots \subset \O b_1 \oplus \O\varpi b_2 \subset \O b_1 \oplus \O b_2 \subset \O b_1 \oplus \O \varpi^{-1} b_2 \subset \cdots $$
and
$$ \cdots \subset \O (b_1+b_2) \oplus \O\varpi b_2 \subset \O (b_1+b_2) \oplus \O b_2 \subset \O (b_1+b_2) \oplus \O \varpi^{-1} b_2 \subset \cdots.$$
One verifies that for any $k \geq 0$, $\O b_1 \oplus \O\varpi^{-k} b_2 = \O (b_1+b_2) \oplus \O\varpi^{-k} b_2$. Thus the two chains of lattices overlap on the half-infinite path with vertices $[\O b_1 \oplus \O\varpi^{-k} b_2] = [\O (b_1+b_2) \oplus \O\varpi^{-k} b_2]$, $k \geq 0$, but become separate afterwards (Figure \ref{fig-two-apts}).  
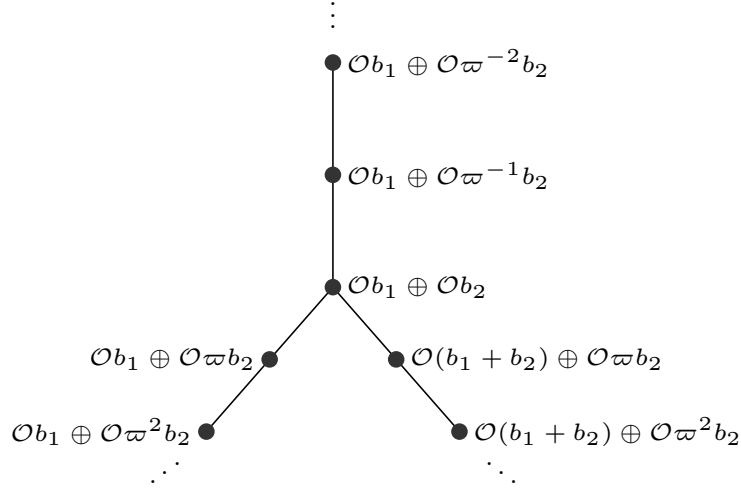
\begin{figure}   
\centering
\begin{tikzpicture}[scale=1.5]
	\begin{axis}[xmin=-5, xmax=6, ymin=-4, ymax=4,
	axis lines*=middle, axis line style={draw=none},
	xlabel=\empty,
	ylabel=\empty,
	xtick=\empty,
	xticklabels=\empty,
	ytick=\empty,
	yticklabels=\empty]
	\addplot[domain={0:1.8}] {-\x};
	\addplot[domain={-1.8:0}] {\x};
	\addplot[color=black] coordinates {(0, 0) (0, 2.8)};;
	\node[black] at (axis cs:0,3.5){\tiny{$\vdots$}};
	\node at (axis cs:0,0) [circle, scale=0.4, draw=black!80,fill=black!80] {};
	\node[black,right] at (axis cs:0,0){\tiny{$\mathcal{O}b_1\oplus\mathcal{O}b_2$}};
	\node at (axis cs:-0.9,-0.9) [circle, scale=0.4, draw=black!80,fill=black!80] {};
	\node[black,left] at (axis cs:-0.9,-0.9){\tiny{$\mathcal{O}b_1\oplus\mathcal{O}\varpi b_2$}};
	\node at (axis cs:0.9,-0.9) [circle, scale=0.4, draw=black!80,fill=black!80] {};
	\node[black,right] at (axis cs:0.9,-0.9){\tiny{$\mathcal{O}(b_1+b_2)\oplus\mathcal{O}\varpi b_2$}};
	\node at (axis cs:1.8,-1.8) [circle, scale=0.4, draw=black!80,fill=black!80] {};
	\node[black,right] at (axis cs:1.8,-1.8){\tiny{$\mathcal{O}(b_1+b_2)\oplus\mathcal{O}\varpi^2 b_2$}};
	\node at (axis cs:-1.8,-1.8) [circle, scale=0.4, draw=black!80,fill=black!80] {};
	\node[black,left] at (axis cs:-1.8,-1.8){\tiny{$\mathcal{O}b_1\oplus\mathcal{O}\varpi^2 b_2$}};
	\node[black] at (axis cs:2.4,-2.2){\tiny{$\ddots$}};
	\node[black] at (axis cs:-2.4,-2.2){\tiny{$\iddots$}};
	
	\node at (axis cs:0,1.4) [circle, scale=0.4, draw=black!80,fill=black!80] {};
	\node[black,right] at (axis cs:0,1.4){\tiny{$\mathcal{O}b_1\oplus\mathcal{O}\varpi^{-1}b_2$}};
	\node at (axis cs:0,2.8) [circle, scale=0.4, draw=black!80,fill=black!80] {};
	\node[black,right] at (axis cs:0,2.8){\tiny{$\mathcal{O}b_1\oplus\mathcal{O}\varpi^{-2}b_2$}};

	\end{axis}
	\end{tikzpicture}

\caption{Two apartments (infinite paths) corresponding to the bases $B=\{b_1, b_2\}$ and $B' = \{b_1+b_2, b_2\}$ overlapping on a half-infinite path.}  \label{fig-two-apts}
\label{fig}
\end{figure}

\subsection{The Tits building of $\GL(r, K)$}  \label{subsec-Tits-bldg}
We now turn to the discussion of Tits building of $\GL(E)$, which is a spherical building associated to the group $\GL(E)$. As in the case of Burhat-Tits building, we describe the Tits building of $\GL(E)$ in two ways: as a simplicial complex $\Delta_\sph(E)$ and, its geometric realization $\B_\sph(E)$ as a topological space. 

The abstract simplicial complex, $\Delta_\sph(E)$, can be described as follows. Its vertices are the proper nonzero vector subspaces of $E$. Two subspaces $F_1$ and $F_2$ are connected if one of them is a subset of the other. The $(k-1)$-simplices of $\Delta_\sph(E)$ are formed by sets of $k$ mutually connected vertices, hence $(k-1)$-simplices correspond to partial flags
$$F_\bullet = (\{0\} \subsetneqq F_1 \subsetneqq \cdots \subsetneqq F_{k} \subsetneqq E).$$ Maximal simplices, i.e. chambers, correspond to complete flags. 
The apartments in $\Delta_\sph(E)$, correspond to frames in $E$. 
A flag $F_\bullet$ is said to be \emph{adapted} to a frame $L = \{L_1, \ldots, L_r\}$ if all the subspaces in the flag are spanned by a subset of the $L_i$. The apartment corresponding to a frame $L$ is the collection of all flags adapted to $L$. 

The geometric realization of the spherical building of $\GL(E)$ can be described using valuations on $E$. This is an analogue of the Goldman-Iwahori realization of the Bruhat-Tits building of $\GL(E)$ as the space of equivalence classes of additive norms on $E$ (Section \ref{subsec-BT-bldg}). This description of Tits building of the general linear group appears in \cite[Section 1.4]{KM-tpb},  the authors are not aware of another reference for this. 

Two valuations $v_1$, $v_2$ are said to be \emph{equivalent} if there exist $\lambda, c \in \R$ with $\lambda > 0$ such that $\lambda v_1 + c = v_2$. 

\begin{definition}[Geometric realization of Tits building]  \label{def-geo-real-Tits}
The geometric realization $\B_\sph(E)$ can be constructed as the set of all equivalence classes of valuations on $E$. The collection of all equivalence classes of valuations that are adapted to a given frame $L$ is the (geometric realization of) the apartment associated to $L$, which we denote by $A_\sph(L)$. Each valuation $v$ adapted to $L$ is uniquely determined by the values $\{v(L_1 \setminus \{0\}), \ldots, v(L_r \setminus \{0\})\}$. it follows that $A_\sph(L)$ can be identified with the $(r-2)$-dimensional sphere $\R^{r-1} / \R_{> 0} \cong S^{r-2}$ (since we are taking equivalence classes of valuations).
\end{definition}


\begin{figure} 
	\begin{tikzpicture}[scale=0.2]

		\draw[-, thick] (-6,0)--(6,0); 
		\draw[thick, domain=-2.6:2.6] plot[smooth]
		(\x,{2*\x});
		\draw[thick, domain=-2.6:2.6] plot[smooth]
		(\x,{-2*\x});
		\draw[thick](0,0) circle (5);
		
		\node at (0,0) {$\bullet$};
		\node at (5,0) {$\bullet$};
		\node at (-5,0) {$\bullet$};
		\node at (2.2,4.4) {$\bullet$};
		\node at (-2.2,4.4) {$\bullet$};
		\node at (2.2,-4.4) {$\bullet$};
		\node at (-2.2,-4.4) {$\bullet$};
		
	\end{tikzpicture}
	\caption{An apartment in the Tits building of $\SL(3)$, a subdivision of circle into $6$ arcs}  
	\label{fig-Coxeter-complex-SL3}

\vspace{.5cm}    
\begin{tikzpicture}[scale=1.5,tdplot_main_coords]

\shadedraw[tdplot_screen_coords,ball color = white] (0,0) circle (1);

\tdplotsetrotatedcoords{0}{90}{90};
\draw[
    tdplot_rotated_coords,
] (1,0,0) arc (0:180:1);
 
\tdplotsetrotatedcoords{90}{53}{90};
\draw[
    tdplot_rotated_coords,
] (1,0,0) arc (0:180:1);

\tdplotsetrotatedcoords{0}{37}{90};
\draw[
    tdplot_rotated_coords,
] (0.92,-0.3,0) arc (0:132:1.1);

\tdplotsetrotatedcoords{0}{45}{90};
\draw[
    tdplot_rotated_coords,
] (0.92,-0.3,0) arc (360:335:1.25);

\tdplotsetrotatedcoords{0}{143}{90};
\draw[
    tdplot_rotated_coords,
] (0.92,-0.3,0) arc (0:132:1.1);

\tdplotsetrotatedcoords{0}{135}{90};
\draw[
    tdplot_rotated_coords,
] (0.92,-0.3,0) arc (360:335:1.25);
 
\draw[thick,-] (-0.83,-0.57,0) -- (0.9,0.4,0) node{};
\draw[thick,-] (-0.9,0.4,0) -- (0.83,-0.57,0) node{};

\end{tikzpicture}
	\caption{An apartment in the Tits building of $\SL(4)$, a triangulation of $2$-sphere}
	\label{fig-Coxeter-complex-SL4}
\end{figure}

\section{Toric vector bundles over a discrete valuation ring}   \label{sec-tvbs-dvr}
As before let $\X = \X_\Sigma$ be a toric scheme over $\Spec(\O)$ where $\Sigma$ is a fan in $N_\R \times \R_{\geq 0}$ with the corresponding polyhedral complex $\Sigma_1$ in $N_\R \times \{1\}$. We recall that $\X$ comes with a distinguished point $x_0$ in the open chart $\U_0$ giving an identification of this open set with the torus $T_\eta = T_K$. 


\begin{definition}   \label{def-tvb-over-S}
A \emph{toric vector bundle over $\X$} is a $T$-linearized vector bundle $\E$ over $\X$.
A toric vector bundle is the same as a $T$-linearized locally free sheaf of constant rank $\O_\X$-modules on $\X$. By a \emph{morphism} of toric vector bundles over $\X_\Sigma$ we mean a torus equivariant morphism of vector bundles. Toric vector bundles over $\X$ and their morphisms form a category. 
\end{definition}

We refer the reader to \cite[Definition 1.6]{Mumford-GIT}
for the definition of a $T$-linearization of a vector bundle on a $T$-scheme (the definition in the above mentioned reference is for line bundles but the same works for vector bundles as well). In short, a $T$-linearization of $\E$ is a lifting of action of $T$ on $\X$ to an action of $T$ on $\E$ by morphisms of vector bundles.

In this section we give a classification of toric vector bundles over $\X$ in terms of piecewise linear maps from the fan $\Sigma$ to $\tilde{\B}(E)$, the total Bruhat-Tits building of $\GL(E)$, for finite dimensional $K$-vector spaces $E$. When $\Sigma$ is complete, this is the same as a piecewise affine map from the polyhedral complex $\Sigma_1$ to the extended Bruhat-Tits building $\tilde{\B}_\aff(E)$.

\subsection{Local equivariant triviality} \label{subsec-local-equiv-trivial}
As in the case of toric vector bundles over a field, the following is a key step in the classification. The proof is similar to the one for toric vector bundles over a field (see \cite[Proposition 2.1.1]{Klyachko}). 
\begin{proposition}   \label{prop-local-equiv-trivial}
Let $\E$ be a toric vector bundle on an affine toric scheme $\U_{\sigma}$ corresponding to a cone $\sigma$ in $N_\R \times \R_{\geq 0}$. Then $\E$ is equivariantly trivial. That is, there exist characters $u_1, \ldots, u_r \in M$ such that $H^0(\U_\sigma, \E) \cong \bigoplus_{i=1}^r R_{\sigma, i}$ where $R_{\sigma, i} \cong R_\sigma$ and $T$ acts on $R_{\sigma, i}$ via character $u_i$.
\end{proposition}
\begin{proof}
It suffices to show there are $T$-weight sections $s_1, \ldots, s_r \in H^0(\U_{\sigma}, \E)$ that freely generate $H^0(\U_{\sigma}, \E)$ as an $R_{\sigma}$-module. 
If the cone $\sigma$ lies in $N_\R \times \{0\}$, there is no special fiber $\X_{\sigma, o}$ and we are in the case of a toric vector bundle over an affine toric variety $\X_{\sigma, \eta}$. Thus, we can assume that $\sigma$ does not lie in $N_\R \times \{0\}$. In this case, 
there is a unique closed $T_s$-orbit $O_\sigma$ in the special fiber $\X_s$. Recall that $T_s = T_\k$ is the torus over the residue field $\k = \O / \m$. Since $O_\sigma$ is an orbit, the vector bundle $\E_{|O_\sigma}$ is equivariantly trivial. That is, we can find a basis of $T$-weight sections $s'_1, \ldots, s'_r \in H^0(O_\sigma, \E_{|O_\sigma})$. By general properties of coherent sheaves on affine schemes, one knows that the $R_{\sigma}$-module homomorphism:
$$j^*: H^0(\U_{\sigma}, \E) \to H^0(O_\sigma, \E_{|O_\sigma}),$$ induced by the inclusion morphism $j: O_\sigma \hookrightarrow \U_{\sigma}$, is surjective. Note that the $R_{\sigma}$-module structure on $H^0(O_\sigma, \E_{|O_\sigma})$ comes from the homomorphism $R_{\sigma} \to \k[O_\sigma]$ determined by the embedding $O_\sigma \hookrightarrow \U_{\sigma}$. Moreover, 
$j^*$ is $T$-equivariant. In other words, $j^*$ is $M$-grading preserving. Thus we can find $T$-weight sections $s_1, \ldots, s_r \in H^0(\U_\sigma, \E)$ 
such that $j^*(s_i) = s'_i$, for all $i$. We claim that these $s_i$ freely generate $H^0(\U_\sigma, \E)$ as an $R_{\sigma}$-module. First we claim that for any $x \in \U_\sigma$, the images of the $s_i$ in the fiber $\E_x$ are linearly independent. This is because the subscheme of $\U_{\sigma}$ where these images are linearly dependent is a closed $T$-invariant subscheme and hence, if non-empty, should contain $O_\sigma$. However, since the $s'_i$ form a basis for $H^0(O_\sigma, \E_{|O_\sigma})$, the degeneracy locus cannot contain $O_\sigma$ as a subscheme and must therefore be empty, which proves the claim. Now from the assumption that $\E$ is locally free we conclude that for all $x \in \X$ the images of $s_i$ in $\E_x$ form a basis. It follows that the $s_i$ are an $R_\sigma$-basis for $H^0(\U_{\sigma}, \E)$ as required. 
\end{proof}

\subsection{Graded piecewise linear maps}
\label{subsec-graded-PL-map}
As usual, let $\Sigma$ be a fan in $N_\R \times \R_{\geq 0}$ and $\Sigma_1$ (respectively $\Sigma_0$) the corresponding polyhedral complex in the affine hyperplane $N_\R \times \{1\}$ (respectively, fan in the hyperplane $N_\R \times \{0\}$). 

In \cite{KM-tpb}, the notion of an (integral) piecewise linear map from a fan in $N_\R$ to the extended Tits building $\tilde{\B}_\sph(E)$ is introduced. These integral piecewise linear maps are exactly the gadgets that classify toric vector bundles over a field. For the sake of completeness we recall this definition here.
\begin{definition}[Piecewise linear map to $\tilde{\B}_\sph(E)$]  \label{def-PL-map}
Let $\Sigma_0$ be a fan in $N_\R \times \{0\}$.
We say that a map $\Phi_0: |\Sigma_0| \to \tilde{\B}_\sph(E)$ is a \emph{piecewise linear} map if the following holds:
\begin{itemize}
\item[(a)] For any cone $\sigma \in \Sigma_0$, there is an extended apartment $\tilde{A}_\sph(\sigma)$ (not necessarily unique) such that the restriction $\Phi_{0, \sigma} :={\Phi_0}_{|\sigma}$ maps $\sigma$ into $\tilde{A}_\sph(\sigma)$.
\item[(b)] For any $\sigma \in \Sigma_0$, $\Phi_{0, \sigma}$ is the restriction of a linear map from $N_\R \times \{0\}$ to $\tilde{A}_\sph(\sigma) \cong \R^r$.
\end{itemize}
We say that $\Phi_0$ is an \emph{integral} piecewise linear map if for each $\sigma \in \Sigma_0$, $\Phi_{0, \sigma}$ is the restriction of an integral linear map from $N_\R \times \{0\}$ to $\tilde{A}_\sph(\sigma)$. In other words, $\Phi_{0, \sigma}$ sends lattice points to lattice points.
\end{definition}

In this section, we define the similar notion of a graded piecewise linear map $\Phi$ from a fan $\Sigma$ in $N_\R \times \R_{\geq 0}$ to the total Bruhat-Tits building $\tilde{\B}(E)$. The main result of the paper (Section \ref{subsec-main}) is that these integral graded piecewise linear maps are exactly the gadgets to classify toric vector bundles on toric schemes over a DVR.

We also define the notion of a piecewise affine map $\Phi_1$ from a polyhedral complex $\Sigma_1$ in $N_\R \times \{1\}$ to the extended Bruhat-Tits building $\tilde{\B}_\aff(E)$. The restriction of a graded piecewise linear map $\Phi$ to $|\Sigma_1|$ (respectively $|\Sigma_0|$)
gives a piecewise affine map $\Phi_1$ (respectively a piecewise linear map $\Phi_0$). In fact, the information of $\Phi$ is equivalent to the information of $(\Phi_0, \Phi_1)$ (see Proposition \ref{prop-Phi-0-1} below). 

Recall that when $\Sigma$ in a complete fan in $N_\R \times \R_{\geq 0}$, the fan $\Sigma_0$ can be recovered from the polyhedral complex $\Sigma_1$ as its recession fan. Thus, in this case, the information of $\Phi$ is the same as the information of $\Phi_1$ alone, and $\Phi_0$ can be recovered from $\Phi_1$ as its linear part.  

\begin{definition}[Graded piecewise linear map to $\tilde{\B}(E)$] \label{def-graded-PL-map}
Let $\Sigma$ be a fan in the upper half-space $N_\R \times \R_{\geq 0}$.
We say that a map $\Phi: |\Sigma| \to \tilde{\B}(E)$ is a \emph{graded piecewise linear} map if the following holds:
\begin{itemize}
\item[(a)] For any cone $\sigma \in \Sigma$, there is an extended apartment $\tilde{A}(\sigma)$ (not necessarily unique) such that the restriction $\Phi_\sigma := \Phi_{|\sigma}$ maps $\sigma$ into $\tilde{A}(\sigma)$.
\item[(b)] For any $\sigma \in \Sigma$, $\Phi_\sigma$ is the restriction of a linear map from $N_\R \times \R_{\geq 0}$ to $\tilde{A}(\sigma) \cong \R^r \times \R_{\geq 0}$.
\item[(c)] For any $m \geq 0$, the map $\Phi$ sends $|\Sigma_m| := |\Sigma| \cap (N_\R \times \{m\})$ to $\tilde{\B}_m(E)$. That is, it sends a point of level $m$ in the fan to a point of level $m$ in the extended building. 
\end{itemize}
We say that $\Phi$ is \emph{integral} if for each $\sigma \in \Sigma$, $\Phi_\sigma$ is the restriction of an integral linear map from $N_\R \times \R_{\geq 0}$ to $\tilde{A}(\sigma)$. 
\end{definition}

We recall that a function $\Phi: N_\R \times \{1\} \to \R^r$ is an \emph{affine map} if $\Phi$ is a linear map followed by a translation. In other words, 
there is a linear map $\Phi_0: N_\R \times \{0\} \to \R^r$ and $b \in \R^r$ such that for all $x \in N_\R$ we have: 
$$\Phi(x, 1) = \Phi_0(x, 0) + b.$$
The linear map $\Phi_0$ is called the \emph{linear part} of $\Phi$. We say that $\Phi$ is an \emph{integral} affine function if $\Phi(N \times \{1\}) \subset \Z^r$. This means that the entries of the matrix of $\Phi_0$, with respect to a $\Z$-basis for $N$ and the standard basis for $\Z^r$, are integers, and moreover, $b \in \Z^r$.

\begin{definition}[Piecewise affine map to $\tilde{\B}_\aff(E)$]  \label{def-PA-map}
Let $\Sigma_1$ be a polyhedral complex in the affine space $N_\R \times \{1\}$.
A map $\Phi_1: |\Sigma_1| \to \tilde{\B}_\aff(E)$ is a \emph{piecewise affine} map if the following holds:
\begin{itemize}
\item[(a)] For any polyhedron $\Delta \in \Sigma_1$, there is a (not necessarily unique) extended apartment $\tilde{A}_\aff(\Delta)$ such that the restriction $\Phi_{1, \Delta} := {\Phi_1}_{|\Delta}$ maps $\Delta$ into $\tilde{A}_\aff(\Delta)$.
\item[(b)] For any $\Delta \in \Sigma_1$, $\Phi_{1, \Delta}$ is the restriction of an affine map from $N_\R \times \{1\}$ to the affine space $\tilde{A}_\aff(\Delta)$.
\end{itemize}
We say that $\Phi$ is an \emph{integral} piecewise affine map if
for each $\Delta \in \Sigma_1$, $\Phi_{1, \Delta}$ is the restriction of an integral affine map from $N_\R \times \{1\}$ to $\tilde{A}_\aff(\Delta)$.
\end{definition}

Let $\Phi: |\Sigma| \to \tilde{\B}(E)$ be a graded piecewise linear map. Then the restriction $\Phi_0: |\Sigma_0| \to \tilde{\B}_0(E) = \tilde{\B}_\sph(E)$ of $\Phi$ to $|\Sigma_0|$ is a piecewise linear map in the sense of Definition \ref{def-PL-map}. Moreover, the restriction $\Phi_1: |\Sigma_1| \to \tilde{\B}_1(E) = \tilde{\B}_\aff(E)$ of $\Phi$ to $|\Sigma_1| := |\Sigma| \cap (N_\R \times \{1\})$ is a piecewise affine map in the sense of Definition \ref{def-PA-map}.

One can recover the graded piecewise linear map $\Phi$ from $\Phi_0$ and $\Phi_1$ in an obvious manner. For $(x, m) \in |\Sigma| \subset N_\R \times \R_{\geq 0}$ we have:
\begin{equation}  \label{equ-Phi-0-1}
\Phi(x, m) = 
\begin{cases}
m\Phi_1(x/m, 1) & \quad m \neq 0,\\
\Phi_0(x, 0) & \quad m = 0.\\
\end{cases}
\end{equation}

One verifies the following.
\begin{proposition}  \label{prop-Phi-0-1}
The equation \eqref{equ-Phi-0-1} gives a one-to-one correspondence between the graded piecewise linear maps $\Phi: |\Sigma| \to \tilde{\B}(E)$ and pairs $(\Phi_0, \Phi_1)$ of piecewise linear maps $\Phi_0: |\Sigma_0| \to \tilde{\B}_\sph(E)$ and piecewise affine maps $\Phi_1: |\Sigma_1| \to \tilde{\B}_\aff(E)$ that are compatible in the following sense: For any polyhedron $\Delta \in \Sigma_1$, the linear part of $\Phi_1$ on the recession cone of $\Delta$ agrees with $\Phi_0$. In particular, when the fan $\Sigma$ is complete, the piecewise linear map $\Phi_0$ is determined by the piecewise affine map $\Phi_1$ as its linear part.
\end{proposition}

\begin{remark}   \label{rem-linear-part}
The above definitions extend to arbitrary affine buildings.
\end{remark}

\begin{remark}  \label{rem-pa-apt}
Let $\Phi_1: |\Sigma_1| \to \tilde{\B}(E)$ be an integral piecewise affine map. For any polyhedron $\Delta \in \Sigma_1$, the requirement that $\Phi_{1, \Delta} := {\Phi_1}_{|\Delta}: \Delta \to \tilde{A}_\aff(\Delta)$ is an integral affine map means that there exists a $K$-basis $B_\Delta = \{b_{\Delta, 1}, \ldots, b_{\Delta, r}\}$ for $E$ and $\{u_{\Delta, 1}, \ldots, u_{\Delta, r}\} \subset M$ such that for any $x \in \sigma$, $\Phi_{1, \Delta}(x)$, as an additive norm on $E$, is given by
\begin{equation}   \label{equ-Phi-sigma} 
\Phi_{1, \Delta}(x)\left(\sum_i \lambda_i b_{\Delta, i}\right) = \min\{ \val(\lambda_i) + \langle u_{\Delta, i}, x \rangle \mid i=1, \ldots, r\}.
\end{equation}    
\end{remark}


\begin{definition}[Morphism of graded piecewise linear maps]   \label{def-morphism-pa-map}
Let $F: E \to E'$ be a linear transformation between finite dimensional $K$-vector spaces $E$, $E'$ and let $\Phi: |\Sigma| \to \tilde{\B}(E)$ and $\Phi': |\Sigma| \to \tilde{\B}(E')$ be two graded piecewise linear maps. We say that $F$ gives a \emph{morphism} $\Phi \to \Phi'$ if for any $(x, m) \in |\Sigma|$ we have: $$\Phi(x, m) \preceq F^*(\Phi'(x, m)),$$ 
(see Definition \ref{def-pull-back} and cf. Proposition \ref{prop-F-v-v'}).
\end{definition}

For a fixed fan $\Sigma$ in $N_\R \times \R_{\geq 0}$, the collection of graded piecewise linear maps $\Phi:|\Sigma| \to \tilde{\B}(E)$, for finite dimensional $K$-vector spaces $E$, together with the above notion of morphism, form a category. 

\subsection{Main theorems}   \label{subsec-main}
We break our classification of toric vector bundles over a DVR into a few theorems (Theorems \ref{th-main}, \ref{th-main2}, \ref{th-main3}). We present the proofs later in this section. The proofs involve using the notion of \emph{ceiling} of an additive norm (Definition \ref{def-ceiling}). 

As usual, let $\Sigma$ be a fan in $N_\R \times \R_{\geq 0}$ with $\X = \X_\Sigma$ the corresponding toric scheme. Recall that $x_0$ is the distinguished point in the open orbit in $\U_0$ in the generic fiber $\X_\eta \cong T_\eta$. For a toric vector bundle $\E$ over $\X_\Sigma$, we let $E=\E_{x_0}$. It is an $r$-dimensional $K$-vector space where $r=\rank(\E)$.
\begin{theorem}    \label{th-main}
There is a one-to-one correspondence $\E \mapsto \Phi_\E$ between the isomorphism classes of toric vector bundles over a toric scheme $\X=\X_\Sigma$ and the integral graded piecewise linear maps from $|\Sigma|$ to $\tilde{\B}(E)$, the total Bruhat-Tits buildings of general linear groups of finite dimensional $K$-vector spaces $E$.
\end{theorem}


If $\E$ is a toric vector bundle over a toric scheme $\X=\X_\Sigma$ then the generic fiber $\X_\eta$ is a usual toric variety over the field $K$ and $\E_{|\X_\eta}$ is a usual toric vector bundle on it. The next theorem describes the piecewise linear map (equivalently, the Klyachko filtrations) associated to the toric vector bundle $\E_{|\X_\eta}$. 
\begin{theorem}  \label{th-main2}
Given a toric vector bundle $\E$ with associated integral graded piecewise linear map $\Phi = \Phi_\E: |\Sigma| \to \tilde{\B}(E)$, the toric vector bundle $\E_{|\X_\eta}$ over the generic fiber $\X_\eta$, which is a toric vector bundle over the field $K$, corresponds to the piecewise linear map $\Phi_0: |\Sigma_0| \to \tilde{\B}_\sph(E)$, obtained by restricting $\Phi$ to $|\Sigma_0|$ (see \cite[Example 2.8]{KM-tpb}). 
\end{theorem}


\begin{theorem}    \label{th-main3}
The correspondence $\E \mapsto \Phi_\E$ gives
an equivalence of categories between the category of toric vector bundles over the toric scheme $\X_\Sigma$ and the category of integral graded piecewise linear maps from $|\Sigma|$ to $\tilde{\B}(E)$, for finite dimensional $K$-vector spaces $E$.
\end{theorem}

When $\Sigma$ is a complete fan, Theorem \ref{th-main3} can be restated in terms of piecewise affine maps. As before, for a graded piecewise linear map $\Phi: |\Sigma| \to \tilde{\B}(E)$ let $\Phi_1: |\Sigma_1| \to \tilde{\B}_\aff(E)$ denote its restriction to $|\Sigma_1|$ which is a piecewise affine map. Recall that in this case, $\Phi_0$ is the linear part of $\Phi_1$.
\begin{corollary}
Let $\Sigma$ be a complete fan in $N_\R \times \R_{\geq 0}$. Then the correspondence $\E \mapsto \Phi_{\E, 1}$ gives
an equivalence of categories between the category of toric vector bundles $\E$ over the toric scheme $\X_\Sigma$ and the category of integral piecewise affine maps $\Phi_1$ from $|\Sigma_1|$ to $\tilde{\B}_\aff(E)$, for finite dimensional $K$-vector spaces $E$.
\end{corollary}


\begin{proof}[Proofs of Theorems \ref{th-main} and \ref{th-main2}]
Recall from Section \ref{subsec-toric-schemes} that, for each $\sigma \in \Sigma$, we have the affine toric scheme $\U_\sigma = \Spec(R_{\sigma})$ where:
\begin{equation}   \label{equ-S-sigma}
R_{\sigma} = \O[\chi^u \varpi^k \mid (u, k) \in \sigma^\vee \cap \tilde{M}] \subset K[T_\eta].
\end{equation}
In particular, $R_0 = K[T_\eta]$ is the coordinate ring of the torus $T_\eta \cong (K^\times)^r$, isomorphic to the Laurent polynomial ring over $K$ in $r$ indeterminates.
The affine toric schemes $\U_{\sigma}$, $\sigma \in \Sigma$, glue together to form the toric scheme $\X = \X_\Sigma$. 
The distinguished point on the scheme $\X_0 = \Spec(K[T_\eta])$ is $x_0 =  1 \in T_\eta$.

Let $\E$ be a toric vector bundle on $\X_\Sigma$ of rank $r$. First we see how to associate to $\E$ an integral graded piecewise linear map $\Phi=\Phi_\E: |\Sigma| \to \tilde{\B}(E)$. The toric vector bundle $\E$  is equivalent to the data of a $T$-linearized sheaf $\F$ of locally free $\O_\X$-modules of rank $r$ on $\X$. Let $E = \E_{x_0}$ which is an $r$-dimensional $K$-vector space. For each cone $\sigma \in \Sigma$, consider the $M$-graded $R_{\sigma}$-module $\F_{\sigma} := H^0(\U_{\sigma}, \E_{|\U_{\sigma}})$. 
Since $\E$ is trivial over the open orbit $\X_0 = T_\eta$, we can identify $\E_{|\X_0}$ with $T_\eta \times E$ and hence $\F_0$ with $K[T_\eta] \otimes_K E$.
Then, for each $\sigma \in \Sigma$, the restriction of $\F_\sigma$ to $\X_0$ gives us an embedding $$\F_{\sigma} \hookrightarrow \F_0 = K[T_\eta] \otimes_K E.$$ 
(From now on, for the ease of notation, we may drop the tensor product sign and write $ab$ in place of $a \otimes b$.) Consider the decomposition of $\F_\sigma$ into $M$-homogeneous components: $$\F_{\sigma} = \bigoplus_{u \in M} \F_{\sigma, u},$$ where $\F_{\sigma, u} \subset \chi^{u} E \subset K[T_\eta] \otimes_K E$ is the $u$-homogeneous component of $\F_{\sigma}$ (note that the $u$-homogeneous component is the $(-u)$-weight space in $\F_\sigma$). Since the $u$-homogeneous component in $K[T_\eta] \otimes_K E$ is $\chi^{u} E$, we can write:
$$\F_{\sigma, u} = \chi^{u} \Lambda_{\sigma, u},$$ for some $\O$-submodule $\Lambda_{\sigma, u} \subset E$. 

On the other hand, by the local equivariant triviality (Proposition \ref{prop-local-equiv-trivial}) the $M$-graded $R_{\sigma}$-module $\F_{\sigma}$ is 
free of rank $r$ and we can write: 
\begin{equation} \label{equ-F-sigma-free}
\F_{\sigma} = \bigoplus_{i=1}^r R_{\sigma} \cdot s_{\sigma, i},
\end{equation} 
where $s_{\sigma, i}$ is an $M$-homogeneous section whose $M$-degree we denote by $u_{\sigma, i}$, that is, the image of $s_{\sigma, i}$ in $K[T_\eta] \otimes_K E$ lands in $\chi^{u_{\sigma, i}}E$.
Let us write:
$$s_{\sigma, i} = \chi^{u_{\sigma, i}} b_{\sigma, i}$$ for some $b_{\sigma, i} \in E$. Since \eqref{equ-F-sigma-free} is a direct sum, the set of vectors $B_\sigma = \{b_{\sigma, 1}, \ldots, b_{\sigma, r}\}$ forms a $K$-vector space basis for $E$. We thus have:
 \begin{equation}  \label{equ-F-sigma}
 \F_{\sigma} = \bigoplus_{i=1}^r \bigoplus_{(u',k') \in \sigma^\vee \cap \tilde{M}} \O \cdot\chi^{u'} \varpi^{k'} \, \chi^{u_{\sigma, i}} b_{\sigma, i}.  
 \end{equation}
To compute the $u$-homogeneous component $\F_{\sigma, u}$ in \eqref{equ-F-sigma} we should sum over all $(u', k')$ such that $u' = u - u_{\sigma, i}$ and $(u', k') \in \sigma^\vee \cap \tilde{M}$. This means that $$\langle (u - u_{\sigma, i}, k'), y \rangle \geq 0,~~ \forall y \in \sigma.$$ That is,
\begin{equation}   \label{equ-Lambda-sigma-u}
\Lambda_{\sigma, u} = \sum_{i=1}^r \sum_{k': (u - u_{\sigma, i}, k') \in \sigma^\vee \cap \tilde{M}} \O\, \varpi^{k'} b_{\sigma, i}.
\end{equation}
We note that the condition $(u - u_{\sigma, i}, k') \in \sigma^\vee$ is equivalent to $(u - u_{\sigma, i}, k') \in \rho^\vee$, for all $\rho \in \sigma(1)$. where $\sigma(1)$ denotes the set of rays in $\sigma$. Thus, from \eqref{equ-Lambda-sigma-u}, we have:   $$\Lambda_{\sigma, u} = \bigcap_{\rho \in \sigma(1)} \Lambda_{\rho, u}.$$ Hence we can conclude that the sheaf $\F$ is completely determined by the $\O$-modules $\Lambda_{\rho, u}$, for $\rho \in \Sigma(1)$ and $u \in M$.
Let $\rho$ be a ray in $\Sigma(1)$. We consider two cases:

\textbf{Case 1:} The ray $\rho$ does not lie in $N_\R \times \{0\}$. Let $(\nu, 1)$ be the unique point of intersection of $\rho$ and $N_\R \times \{1\}$. Note that $\nu$ is a rational point but in general may not be a lattice point. Then: 
\begin{align}
(u-u_{\sigma, i}, k') \in \rho^\vee \Leftrightarrow& ~\langle u - u_{\sigma, i}, \nu  \rangle + k' \geq 0 \\
\Leftrightarrow&~ k' \geq \langle u_{\sigma, i} - u, \nu \rangle \\
\Leftrightarrow&~ k' \geq \lceil \langle u_{\sigma, i} - u, v \rangle \rceil. \label{equ-k'}  
\end{align}
We recall that $\lceil r \rceil$ denotes the ceiling of a real number $r$, that is, the smallest integer greater than or equal to $r$.
From \eqref{equ-k'} we then get:
\begin{equation}  \label{equ-L-rho}
\Lambda_{\rho, u} = \bigoplus_{i=1}^r \O \, \varpi^{\lceil \langle u_{\sigma, i} - u, \nu \rangle \rceil} b_{\sigma, i},
\end{equation}
Since $B_\sigma = \{b_{\sigma, 1}, \ldots, b_{\sigma, r}\}$ is a $K$-basis for $E$ we see that, for any $u \in M$, the $\O$-module $\Lambda_{\rho, u} \subset E$ is of full rank, i.e. it is an $\O$-lattice. 

\textbf{Case 2:} The ray $\rho$ lies in $N_\R \times \{0\}$. Let $(\nu, 0)$ be a primitive vector in $\rho$. Then: 
\begin{equation}  \label{equ-k'-2}
(u - u_{\sigma, i}, k')  \in \rho^\vee \Leftrightarrow ~\langle u - u_{\sigma, i}, \nu  \rangle \geq 0. 
\end{equation}
Noting that $K = \bigcup_{k' \in \Z} \O \varpi^{k'}$ we get:
\begin{equation}  \label{equ-L-rho-2}
\Lambda_{\rho, u} = \bigoplus_{i: \langle u, \nu  \rangle \geq \langle u_{\sigma, i}, \nu  \rangle} K \cdot b_{\sigma, i}.
\end{equation}
In other words, $\Lambda_{\rho, u}$ depends only on the integer $\langle u, \nu \rangle$, that is, on the class $[u]$ of $u$ in $M_{\rho} := M / (\rho^\perp \cap M) \cong \Z$. Let us write: 
\begin{equation} \label{equ-L-rho-3}
E^\rho_j = \Lambda_{\rho, u}
\end{equation}
where $j = \langle -u, \nu \rangle$ (the reason for $-u$ instead of $u$ is to make the $\{E^\rho_j \}_{j \in \Z}$ a decreasing filtration and hence match with the convention in \cite{Klyachko}). We observe that for rays $\rho$ that lie in $\sigma \cap (N_\R \times \{0\})$, the equation \eqref{equ-L-rho-2} for the filtrations $\{E^\rho_j\}_{j \in \Z}$ is the Klyachko compatibility condition. It follows, exactly as in the classification of toric vector bundles over a field (Theorem \ref{th-Klyachko}) that $\F$ restricted to the generic fiber $\X_\eta$ is determined, up to an equivariant isomorphim, by the decreasing $\Z$-filtrations $\{E^\rho_j\}_{j \in \Z}$, $\rho \in \Sigma_0(1)$.

Now we are ready to construct the graded piecewise linear map $\Phi$ associated to the toric vector bundle $\E$. Let $\Delta \in \Sigma_1$ be a polyhedron with $\sigma \in \Sigma$ the corresponding cone over it. In this case, we denote the basis $B_\sigma = \{b_{\sigma, i}\}$ by $B_\Delta = \{b_{\Delta, i}\}$. Also, if $\nu$ is a vertex in $\Delta$ with corresponding ray $\rho \in \sigma(1)$, we denote the $\O$-module $\Lambda_{\rho, u}$ by $\Lambda_{\nu, u}$.

Let $\tilde{A}_\aff(\Delta)$ be the extended apartment in the extended building $\tilde{\B}_\aff(E)$ corresponding to the basis $B_\Delta$. Recall that it consists of additive norms that are adapted to the basis $B_\Delta$ (see Section \ref{subsec-BT-bldg}). 
Note that choice of a basis $B_\Delta$ gives an identification of the extended apartment $\tilde{A}_\aff(\Delta)$ with $\R^r$. We define the affine map $\Phi_{1, \Delta}: \Delta \to \tilde{A}_\aff(\Delta)$ to be the linear map given by $\Phi_{1, \Delta}(x) = (\langle u_{\sigma, 1}, x \rangle, \ldots, \langle u_{\sigma, r}, x \rangle).$
In terms of additive norms, $\Phi_{1, \Delta}(x): E \to \overline{\R}$ is the additive norm given by (Remark \ref{rem-pa-apt}):
\begin{equation}   \label{equ-Phi-Delta-1}
\Phi_{1, \Delta}(x)\left(\sum_i \lambda_i b_{\Delta, i}\right) = \min\{ \val(\lambda_i) + \langle u_{\sigma, i}, x \rangle \mid i=1, \ldots, r\}.
\end{equation}
For a vertex $\nu$ of the polyhedron $\Delta$ and $u \in M$, the lattice $\Lambda_{\nu, u}$ can be recovered from $\Phi_{1, \Delta}$ as
\begin{equation}   \label{equ-Phi-Delta}
	\Lambda_{\nu, u} = \bigoplus_{i=1}^r \O \varpi^{\lceil \Phi_{1, \Delta}(\nu)(b_{\Delta, i}) - \langle u, \nu \rangle \rceil} b_{\Delta, i}.
\end{equation}
We also need to check that the affine map $\Phi_{1, \Delta}$ contains the data of filtrations $\{E^\rho_j\}_{j \in \Z}$, where the $\rho \in N_\R \times \{0\}$ are horizontal rays in the cone $\sigma$ (see \eqref{equ-L-rho-2} and \eqref{equ-L-rho-3}). We see that the filtrations $\{E^\rho_j\}_{j \in \Z}$ can be recovered from the linear part $\Phi_{0, \Delta}$ of $\Phi_{1, \Delta}$. 
For any $(x, 0)$ in the linear span of the recession cone of $\Delta$, we can compute the value of the valuation $\Phi_{0, \Delta}(x, 0)$ on a vector $e = \sum_i \lambda_i b_{\Delta, i}$ by:
\begin{align*}
\Phi_{0, \Delta}(x, 0)\left(\sum_i \lambda_i b_{\Delta, i}\right) &= \lim_{t \to \infty} \frac{\min\{ \val(\lambda_i) + \langle u_{\Delta, i}, tx \rangle \mid i=1, \ldots, r\}}{t}, \\
&= \min\{\langle u_{\Delta, i}, x \rangle \mid i=1, \ldots, r\}. 
\end{align*}
Let $(\nu, 0)$ be the primitive vector for a ray $\rho$ in $\sigma$, the cone over $\Delta$.
If we let $x=-\nu$, the valuation $\Phi_{0, \Delta}(-\nu, 0)$ is given by: 
$$\Phi_{0, \Delta}(-\nu, 0)\left(\sum_i \lambda_i b_{\Delta, i}\right) = \min\{\langle u_{\Delta, i}, -\nu \rangle \mid i=1, \ldots, r \},$$ and its corresponding 
filtration $\{E_j\}_{j \in \Z}$ is: $$E_j = \textup{span}_K\{ b_{\Delta, i} \mid j \leq \langle u_{\Delta, i}, -\nu \rangle \}.$$
{But this is the same as \eqref{equ-L-rho-2}, the filtration associated to the ray $\rho \in \Sigma_0$, as required.}

It remains to show that the $\Phi_{1, \Delta}$, $\Delta \in \Sigma_1$, glue together to give a piecewise affine map $\Phi_1: |\Sigma_1| \to \tilde{\B}_\aff(E)$. To this end, we need to show that for all $\Delta, \Delta' \in \Sigma_1$, $\Phi_{1, \Delta}$ and $\Phi_{1, \Delta'}$ coincide on the intersection $\Delta \cap \Delta'$. This is the content of the next lemma. If the vertices of $\Delta \cap \Delta'$ are lattice points, the lemma is easy to show because we can remove the ceiling function in the expression for $\Lambda_{u, \nu}$ in \eqref{equ-Phi-Delta}. In general, the presence of the ceiling function adds extra details to the argument. 
\begin{lemma}  \label{lem-gluing}
Let $\nu$ be a vertex of the convex polyhedron $\Delta \cap \Delta'$. Then the additive norms $\Phi_{1, \Delta}(\nu)$ and $\Phi_{1, \Delta'}(\nu)$ coincide. 
\end{lemma}
\begin{proof}
 Let $B_\Delta = \{b_{\Delta, i}\}$ and $B_{\Delta'} = \{b_{\Delta', i}\}$ be the bases corresponding to $\Delta$ and $\Delta'$ respectively. We note that, by \eqref{equ-Phi-Delta}, the additive norm corresponding to $\Lambda_{\nu, u}$ is:
\begin{equation}   \label{equ-Phi-Phi'-1}
\lceil \Phi_{1, \Delta}(\nu)(\cdot) - \langle u, \nu \rangle \rceil = \lceil \Phi_{1, \Delta'}(\nu)(\cdot) - \langle u, \nu \rangle \rceil.
\end{equation}
By contradiction, suppose $\Phi_{1, \Delta}(\nu) \neq \Phi_{1, \Delta'}(\nu)$. Then for some $e \in E$, $\Phi_{1, \Delta}(\nu)(e) \neq \Phi_{1, \Delta'}(\nu)(e)$. Let $e = \sum_i \lambda_i b_{\Delta, i} = \sum_i \lambda'_i b_{\Delta', i}$. From \eqref{equ-Phi-Delta} we see that, after a reordering of basis elements if necessary, there is an index $j$ such that:
$$\Phi_{1, \Delta}(e) = \val(\lambda_j) + \langle u_{\sigma, j}, \nu \rangle \neq  \val(\lambda'_j) + \langle u_{\sigma', j}, \nu \rangle = \Phi_{1, \Delta'}(e).$$
On the other hand, in \eqref{equ-Phi-Phi'-1}, if we let $u=u_{\Delta, j}$ we get:
$$\lceil \langle u_{\Delta', j}, \nu \rangle - \langle u_{\Delta, j}, \nu \rangle \rceil = \val(\lambda_j) - \val(\lambda'_j).$$ 
Also, in \eqref{equ-Phi-Phi'-1}, if we let $u = u_{\Delta', j}$, in the same way we obtain:  
$$\lceil \langle u_{\Delta, j}, \nu \rangle - \langle u_{\Delta', j}, \nu \rangle \rceil = \val(\lambda'_i) - \val(\lambda_j).$$ 
It follows that $\langle u_{\Delta, j}, \nu \rangle - \langle u_{\Delta', j}, \nu \rangle$ is an integer. Thus, $\Phi_{1, \Delta}(\nu)(e) - \Phi_{1, \Delta'}(\nu)(e)$ is an integer. This together with \eqref{equ-Phi-Phi'-1} implies that $\Phi_{1, \Delta}(\nu)(e) = \Phi_{1, \Delta'}(\nu)(e)$. The contradiction proves the claim. 
\end{proof}

We have shown that $\Phi_{1, \Delta}$ and $\Phi_{1, \Delta'}$ coincide on the vertices of the polyhedron $\Delta \cap \Delta'$. We need to show that the linear parts of $\Phi_{1, \Delta}$ and $\Phi_{1, \Delta'}$ also coincide on the recession cone of $\Delta \cap \Delta'$. Let $\rho \in \Sigma_0$ be a ray in the recession cone of the polyhedron $\Delta \cap \Delta'$ with primitive vector $(\nu, 0)$. As we saw above, both valuations  $\Phi_{0, \Delta}(\nu, 0)$ and $\Phi_{0, \Delta'}(\nu, 0)$ on $E$ correspond to the filtration $\{ E^\rho_j \}$ on $E$ and hence coincide. 

Now $\Phi_{1, \Delta}$ and $\Phi_{1, \Delta'}$ are two affine maps on $\Delta \cap \Delta'$ that coincide on the vertices of $\Delta \cap \Delta'$ and also have the same linear parts on $\Delta \cap \Delta'$. This implies that they coincide everywhere on $\Delta \cap \Delta'$ as required. Hence the $\Phi_{1, \Delta}$, $\Delta \in \Sigma_1$, glue together to give a piecewise affine map from $|\Sigma_1|$ to the extended building $\tilde{\B}_\aff(E)$. Also, exactly as in \cite{KM-tpb, Klyachko}, we construct a piecewise linear map $\Phi_0: |\Sigma_0| \to \tilde{\B}_\sph(E)$.

Note that the above arguments show that $\Phi_0$ and $\Phi_1$ are uniquely determined by the vector bundle $\E$. This is because, as shown above, the $\Phi_{1, \Delta}$ and $\Phi_{0, \rho}$ (hence $\Phi_1$ and $\Phi_0$) are determined by the lattices $\Lambda_{\rho, u}$, $u \in M$, (for rays $\rho$ in $\Sigma$ that intersects $N_\R \times \{1\}$), and the filtrations $\Lambda_\rho$ (for rays $\rho$ that lie in $N_\R \times \{0\}$).

Moreover, the above shows that $\Phi_1$ and $\Phi_0$ are compatible in the sense of Proposition \ref{prop-Phi-0-1} and thus they determine an integral graded piecewise linear map $\Phi: |\Sigma| \to \tilde{\B}(E)$. In conclusion $\Phi$ determines the sheaf $\F$, and $\Phi_0$ gives the Klyachko's data for the restriction of the sheaf $\F$ to the generic fiber $\X_\eta$. 

Conversely, let $\Phi: |\Sigma| \to \tilde{\B}_\aff(E)$ be an integral graded piecewise linear map. One verifies that reversing the construction gives a toric vector bundle $\E_\Phi$ over the toric scheme $\X_{\Sigma}$.
\end{proof}

Next we prove the claim about the equivalence of categories. Again, if all the vertices in the the polyhedral complex $\Sigma_1$ are lattice points, then arguments are easier. In general, we need extra details involving the ceiling function. 

\begin{proof}[Proof of Theorem \ref{th-main3}]
A $T$-equivariant morphism of vector bundles is the same as a $T$-equivariant morphism of the corresponding locally free sheaves of modules. 
Let $\E$, $\E'$ be toric vector bundles on $\X = \X_\Sigma$ with the corresponding sheaves of $\O_\X$-modules $\F$ and $\F'$ respectively. Let $F: \F \to \F'$ be a $T$-equivariant morphism. 
Since $\U_0 = T_\eta$ is open in $\X$, $F$ is determined by its restriction to $\U_0$. On the other hand, $F$ is $T$-equivariant and hence $F_{|\U_0}$ is determined by its restriction to $E = \E_{x_0}$. We note that $F: \E_{x_0} \to \E'_{x_0}$ is a $K$-linear map. We would like to characterize which $K$-linear maps $F: \E_{x_0} \to \E'_{x_0}$ give rise to $T$-equivariant morphisms $F: \F \to \F'$ of $\O_\X$-modules. In fact, one verifies that a linear map $F$ gives rise to a $T$-equivariant morphism if and only if for any $\sigma \in \Sigma$ and any character $u \in M$ we have:
\begin{equation}    \label{equ-F-Lambda-sigma}
F(\Lambda_{\sigma, u}) \subset \Lambda'_{\sigma, u}.
\end{equation} 
We note that since $\Lambda_{\sigma, u} = \bigcap_{\rho \in \sigma(1)} \Lambda_{\rho, u}$, the above is equivalent to
\begin{equation}   \label{equ-F-Lambda-rho}
F(\Lambda_{\rho, u}) \subset \Lambda'_{\rho, u}, \quad \forall \rho \in \Sigma(1), ~\forall u \in M.
\end{equation}
The next lemma shows that the condition in \eqref{equ-F-Lambda-rho} is equivalent to 
$\Phi$ being dominated by $F^*(\Phi')$. 
This finishes the proof of Theorem \ref{th-main3}
\end{proof}

\begin{lemma}
The following are equivalent:
\begin{itemize}	
\item[(a)] 	For all $\rho \in \Sigma$ and $u \in M$,
$$F(\Lambda_{\rho, u}) \subset \Lambda'_{\rho, u},$$ where $\Lambda_{\rho, u}$ is given by \eqref{equ-L-rho} if $\rho$ intersects $N_\R \times \{1\}$ and is given by \eqref{equ-L-rho-2} otherwise.
\item[(b)] For all $x \in |\Sigma_1|$, $u \in M$ and $e \in E$:
$$\lceil \Phi_1(x)(e) - \langle u, v \rangle \rceil  \leq \lceil \Phi_1'(x)(F(e)) -  \langle u, x \rangle \rceil.$$
\item[(c)] For all $x \in |\Sigma_1|$ and $e \in E$:
$$\Phi_1(x)(e) \leq F^*(\Phi_1'(x))(e) := \Phi_1'(x)(F(e)).$$ 
\end{itemize}
\end{lemma}
\begin{proof}
(a) $\Rightarrow$ (b): Define $\Psi_u(x)(e) = \lceil \Phi_1(x)(e) - \langle u, v \rangle \rceil$ and  $\Psi'_u(x)(e) = \lceil \Phi_1'(x)(F(e)) -  \langle u, x \rangle \rceil$. Then $\Psi_u$ and $\Psi'_u$ are piecewise affine functions on $|\Sigma_1|$. The condition in (a) implies that, for any vertex $v$ in $\Sigma_1$, $\Psi_u(v) \preceq \Psi'(v)$, and moreover the linear part of $\Psi_u$ is less than or equal to the linear part of $\Psi'_u$. These put together imply that $\Psi_u(x) \preceq \Psi'_u(x)$, for all $x \in |\Sigma_1|$. 

(b) $\Rightarrow$ (c): 
Suppose by contradiction that there is $x$ and $e$ such that $\Phi_1(x)(e) > \Phi_1'(x)(F(e))$.
By the piecewise affineness of $\Phi_1'$ we know that, for any polyhedron $\Delta \in \Sigma_1$ containing $x$,  there is a basis $B'_{\Delta} = \{b'_{\Delta, 1}, \ldots, b'_{\Delta, r}\}$ for $E'$ and characters $\{u'_{\Delta, 1}, \ldots, u'_{\Delta, r}\} \subset M$ such that for all $e' = \sum_i \lambda_i b'_{\Delta, i} \in E$ we have:
$$\Phi_1'(x)(e') = \min\{ \val(\lambda_i)+\langle u'_{\Delta, i}, x\rangle \mid i=1, \ldots, r \}.$$
Let $j$ be the index where the minimum above is attained for $e' = F(e)$, thus
\begin{equation}   \label{equ-Phi'-F-e}
\Phi_1'(x)(F(e)) = \val(\lambda_j) + \langle u'_{\Delta, j}, x \rangle.
\end{equation}
Letting $u = u'_{\Delta, j}$, we see from \eqref{equ-Phi'-F-e}  that $\Phi_1'(x)(F(e)) - \langle u, x\rangle \in \Z$. But 
$$\Phi_1(x)(e) - \langle u, x \rangle  > \Phi_1'(x)(F(e)) - \langle u, x \rangle.$$ Since the right-hand side is an integer we see that
$$\lceil \Phi_1(x)(e) - \langle u, x \rangle \rceil > \lceil \Phi_1'(x)(F(e)) - \langle u, x \rangle \rceil,$$ which contradicts the assumption in (b).

(c) $\Rightarrow$ (a):
First suppose for all $x \in |\Sigma|$ and $e \in E$ we have
$\Phi_1(x)(e) \leq F^*(\Phi_1'(x))(e) := \Phi_1'(x)(F(e))$.
Then, for all $u \in M$, $x \in |\Sigma|$, $e \in E$ one has
\begin{align*}
	\Phi_1(x)(e) - \langle x, u \rangle &\leq \Phi_1'(x)(F(e)) - \langle x, u \rangle, \\
	\lceil \Phi_1(x)(e) - \langle x, u \rangle \rceil &\leq \lceil \Phi_1'(x)(F(e)) - \langle x, u \rangle \rceil.
\end{align*}
From Proposition \ref{prop-F-v-v'} we now conclude that $F(\Lambda_{\rho, u}) \subset \Lambda'_{\rho, u}$ as required.
\end{proof}

\begin{example}   \label{ex-PA-map-tree}
Let the base field be the $2$-adic field $\Q_2$. Recall that the Bruhat-Tits building $\B_\aff(\Q_2^2)$ of $\GL(2, \Q_2)$ is the $3$-regular infinite tree depicted in Figure \ref{fig-Tree-Q2-Q3}. Each apartment in the Bruhat-Tits building is a two-sided infinite path in the tree. Each extended apartment is the Cartesian product of a two-sided infinite path and $\R$. 

Consider the canonical model of $\mathbb{P}^1$ over $\Q_2$. Its fan $\Sigma$ is the Cartesian product of fan of $\mathbb{P}^1$ and $\R_{\geq 0}$. The polyhedral complex $\Sigma_1$ in $N_\R \times \{1\}$ consists of three polyhedra $(-\infty, 0]$, $\{0\}$ and $[0, \infty\}$. The generic fiber is $\mathbb{P}^1$ over $\Q_2$ and the special fiber is $\mathbb{P}^1$ over the residue field $\mathbb{F}_2$. 
In Figure \ref{fig-tree-PA-map}, we have marked three one-sided infinite paths $I$, $I'$ and $I''$. We let $\tilde{I}$, $\tilde{I}'$ and $\tilde{I}''$ denote the Cartesian products of $I$, $I'$ and $I''$ with $\R$ respectively.


Consider integral piecewise affine maps $\Phi_1, \Phi_1': \Sigma \to \tilde{\B}_\aff(E)$, where $\Phi_1$ is given by some integral affine maps ${\Phi_1}_{|(-\infty, 0]} :(-\infty, 0] \to \tilde{I}$ and ${\Phi_1}_{|[0, \infty)}: [0, \infty) \to \tilde{I}'$. Similarly, $\Phi_1'$ is given by some integral affine maps ${\Phi_1'}_{|(-\infty, 0]} :[-\infty, 0] \to \tilde{I}$ and ${\Phi_1'}_{|[0, \infty)}: [0, \infty) \to \tilde{I}''$.
The piecewise affine maps $\Phi_1$ and $\Phi_1'$, give rise to toric vector bundles $\E$ and $\E'$, of rank $2$ on the canonical model of $\mathbb{P}^1$. In Example \ref{ex-split-tvb} we will show that $\E$ splits equivariantly into sum of line bundles while $\E'$ does not.
\end{example}

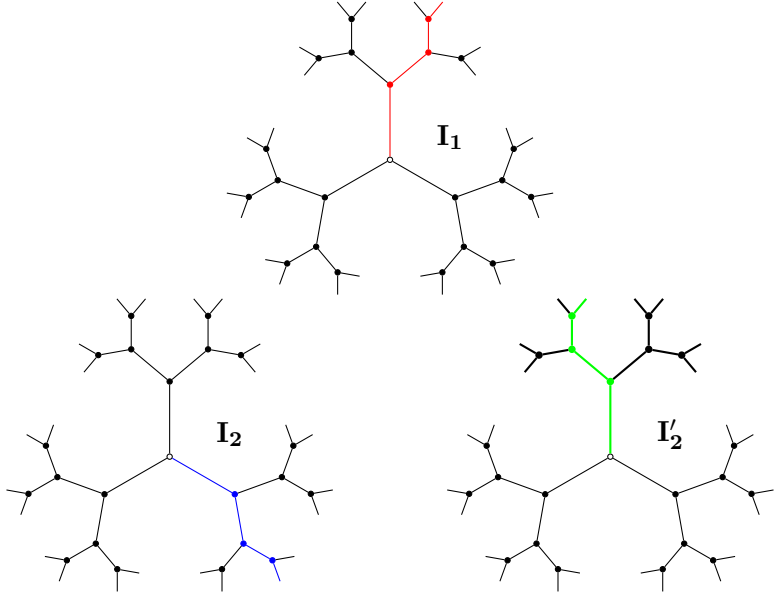
\begin{figure} 
	\centering
	\begin{tikzpicture}[
  grow cyclic,
  level distance=1cm,
  level/.style={
    level distance/.expanded=\ifnum#1>1 \tikzleveldistance/1.5\else\tikzleveldistance\fi,
    nodes/.expanded={\ifodd#1 fill\else fill\fi}
  },
  level 1/.style={sibling angle=120},
  level 2/.style={sibling angle=100},
  level 3/.style={sibling angle=100},
  level 4/.style={sibling angle=80},
  nodes={circle,draw,inner sep=+0pt, minimum size=2pt},
  ]
\path[rotate=-30]
  node {}
  child [black] {
    node {}
    child foreach \cntII in {1,...,2} { 
      node {}
      child foreach \cntIII in {1,...,2} {
        node {}
        child foreach \cntIV in {1,...,2} 
      }
    }
  }
  child [black] {
    node {}
    child foreach \cntII in {1,...,2} { 
      node {}
      child foreach \cntIII in {1,...,2} {
        node {}
        child foreach \cntIV in {1,...,2} 
      }
    }
  } 
  child [red] {
    node {}
    child [red] { 
      node {}
      child[black]  {
        node {}
        child foreach \cntIV in {1,...,2} 
      }
       child[red]  {
        node {}
        child[red]
        child[black]
      }
      }
     child [black] { 
      node {}
      child foreach \cntIII in {1,...,2} {
        node {}
        child foreach \cntIV in {1,...,2} 
      }
    }
  };
  \node[fill=none, draw=none] at (0.8,0.3) {\large{$\mathbf{I_1}$}};
\end{tikzpicture}

\begin{tikzpicture}[
  grow cyclic,
  level distance=1cm,
  level/.style={
    level distance/.expanded=\ifnum#1>1 \tikzleveldistance/1.5\else\tikzleveldistance\fi,
    nodes/.expanded={\ifodd#1 fill\else fill\fi}
  },
  level 1/.style={sibling angle=120},
  level 2/.style={sibling angle=100},
  level 3/.style={sibling angle=100},
  level 4/.style={sibling angle=80},
  nodes={circle,draw,inner sep=+0pt, minimum size=2pt},
  ]
\path[rotate=-30]
  node {}
  child [black] {
    node {}
    child foreach \cntII in {1,...,2} { 
      node {}
      child foreach \cntIII in {1,...,2} {
        node {}
        child foreach \cntIV in {1,...,2} 
      }
    }
  } 
  child [blue] {
    node {}
    child [blue] { 
      node {}
      child[black]  {
        node {}
        child foreach \cntIV in {1,...,2} 
      }
       child[blue]  {
        node {}
        child[blue]
        child[black]
      }
      }
     child [black] { 
      node {}
      child foreach \cntIII in {1,...,2} {
        node {}
        child foreach \cntIV in {1,...,2} 
      }
    }
  }
    child [black] {
    node {}
    child foreach \cntII in {1,...,2} { 
      node {}
      child foreach \cntIII in {1,...,2} {
        node {}
        child foreach \cntIV in {1,...,2} 
      }
    }
  };
  	\node[fill=none, draw=none] at (0.8,0.3) {\large{$\mathbf{I_2}$}};
\end{tikzpicture}
\hspace{0.5in}
\begin{tikzpicture}[
  grow cyclic,
  level distance=1cm,
  level/.style={
    level distance/.expanded=\ifnum#1>1 \tikzleveldistance/1.5\else\tikzleveldistance\fi,
    nodes/.expanded={\ifodd#1 fill\else fill\fi}
  },
  level 1/.style={sibling angle=120},
  level 2/.style={sibling angle=100},
  level 3/.style={sibling angle=100},
  level 4/.style={sibling angle=80},
  nodes={circle,draw,inner sep=+0pt, minimum size=2pt},
  ]
\path[rotate=-30]
  node {}
  child [black] {
    node {}
    child foreach \cntII in {1,...,2} { 
      node {}
      child foreach \cntIII in {1,...,2} {
        node {}
        child foreach \cntIV in {1,...,2} 
      }
    }
  }
  child [black] {
    node {}
    child foreach \cntII in {1,...,2} { 
      node {}
      child foreach \cntIII in {1,...,2} {
        node {}
        child foreach \cntIV in {1,...,2} 
      }
    }
  } 
  child [green!100!blue, thick] {
    node {}
    child [black] { 
      node {}
      child foreach \cntIII in {1,...,2} {
        node {}
        child foreach \cntIV in {1,...,2} 
      }
    }
    child [green!100!blue] { 
      node {}
      child[green!100!blue]  {
        node {}
        child[green!100!blue]
        child[black]
      }
      child[black]  {
        node {}
        child foreach \cntIV in {1,...,2} 
      }
      }
  };
  \node[fill=none, draw=none] at (0.8,0.3) {\large{$\mathbf{I'_2}$}};
\end{tikzpicture}
	\caption{One-sided infinite paths in the tree of $\Q_2$}
\label{fig-tree-PA-map}
\end{figure}

	
	\section{Equivariant splitting of toric vector bundles on toric schemes}   \label{sec-splitting}
In this section we give a simple criterion for splitting of a toric vector bundle into a sum of toric vector bundles in terms of its associated piecewise affine map. It extends the similar criterion for splitting of toric vector bundles over a field (see \cite{Klyachko} as well as \cite[p. 5, A little application]{KM-tpb}).
\begin{definition} \label{def-equiv-splitting}
We say that a toric vector bundle $\E$ \emph{splits equivariantly} if it is equivariantly isomorphic to a direct sum of toric line bundles. 
\end{definition}

\begin{theorem}[Criterion for equivariant splitting]  \label{th-splitting}
A toric vector bundle $\E$ over $\X_\Sigma$ splits equivariantly if and only if the image of the corresponding piecewise linear map $\Phi_\E$ lands in one extended apartment in $\tilde{\B}(E)$. Moreover, if $\Sigma$ is complete (and hence $\X_\Sigma$ is proper) this is equivalent to the image of the piecewise affine map $\Phi_{\E, 1}$ to land in one extended apartment in $\tilde{\B}_\aff(E)$.
	\end{theorem}
\begin{proof}
Let $\E = \underset{i=1}{\overset{r}{\bigoplus}} \L_i$ where the $\L_i$ are toric line bundles on $\X_\Sigma$. As before, the local equivariant triviality (Propositions \ref{prop-toric-vb-over-affine-equiv-trivial} and \ref{prop-local-equiv-trivial}) applied to each toric line bundle $\L_i$ implies that there exists a basis $B=\{b_1, \ldots, b_r\}$ of $E$ with $b_i\in \L_{i|_{x_0}}\subset E$, 
and for any $\sigma \in \Sigma$, there exists a multi-set of weights $\{u_{\sigma, i} \mid i=1, \ldots, r\}$, such that the corresponding piecewise linear map $\Phi_\E$ is then given by (cf. \eqref{equ-Phi-Delta-1}):
$$\Phi_\E(x, m)\left(\sum_i \lambda_i b_i\right) = \min\{ m\val(\lambda_i) + \langle u_{\sigma, i}, x \rangle \mid i=1, \ldots, r\},$$
for any $(x, m) \in \sigma \cap (N_\R \times \{m\})$. This shows that all the level $m$ additive norms $\Phi_\E(x, m)$, are adapted to the basis $B$ and hence the image of $\Phi_\E$ lands in the extended apartment $\tilde{A}(B)$.

Conversely, if there is a basis $B = \{b_i\}$ such that all the additive norms $\Phi_\B(x)$ are adapted to $B$ then, as above, one sees that $\E = \underset{i=1}{\overset{r}{\bigoplus}} \L_i$ where $\L_i$ is the toric line bundle whose sheaf of sections $\F_i$ is given as follows: for any $\sigma \in \Sigma$, $$\F_i(\U_\sigma) = \bigoplus_{(u',k') \in \sigma^\vee \cap \tilde{M}} \O \cdot\chi^{u'} \varpi^{k'} \, \chi^{u_{\sigma, i}} b_{i}.$$ 

To prove the last part, we only need to show that if the image of $\Phi_{\E, 1}$ lands in an (extended) apartment $\tilde{A}_\aff(B)=\tilde{A}_1(B)$ then the image of $\Phi_{\E, 0}$ lands in the (extended) apartment $\tilde{A}_\sph(B)=\tilde{A}_0(B)$. But this follows from the continuity of $\Phi_\E$ and the fact that the intersection of the closure of $\sqcup_{m>0} \tilde{A}_m(B)$ with $\tilde{\B}_0(E)$ is $\tilde{A}_0(B)$.   
\end{proof}

	
\begin{example}   \label{ex-split-tvb}
In Example \ref{ex-PA-map-tree} the image of $\Phi_1$ lands in an apartment, that is, an infinite two-sided path, while the image of $\Phi_2$ does not. 
It follows from Theorem  \ref{th-splitting} that the toric vector bundle $\E_{\Phi_1}$, associated to the piecewise affine map $\Phi_1$, is equivariantly split while $\E_{\Phi_2}$, associated to $\Phi_2$, is not. This is in contrast with the field case: over a field, any toric vector bundle over $\mathbb{P}^1$ splits equivariantly (see \cite[Theorem 6.1.2]{Klyachko} as well as \cite[p. 5, A little application]{KM-tpb}). The example of $\E_{\Phi_2}$ shows that this is not the case over a discrete valuation ring. 
\end{example}

\begin{example}[A Helly's theorem for the tree of $\GL(2, \Q_p)$]	   \label{ex-Helly-GL2}
We observe that a version of Kyachko's Helly's theorem (\cite[Theorem 6.1.2]{Klyachko}) holds for the infinite $(p+1)$-regular tree $T_p$: Let $A$ be a finite collection of vertices in the tree $T_p$ such that any $3$ vertices in $A$ lie on the same path. Then all of the vertices in $A$ lie on the same path. One can give a direct combinatorial proof of the above fact as follows: By induction on $|A|$, it suffices to show if any $k-1$ vertices in $A$ lie on a path then any $k$ vertices also lie on a path. Take vertices $v_i \in A$, $i=1, \ldots, k$. Let $P$ be a path containing the vertices $v_i \in A$, $i=1, \ldots k-1$. We can assume that $v_1$ and $v_{k-1}$ are the first and last vertices in $A$ that appear in $P$. We recall that, in a tree, there is a unique path joining any two given vertices. Now by assumption, $v_1, v_{k-1}, v_k$ lie on the same path $Q$. The uniqueness of path joining $v_1$ and $v_{k-1}$ then implies that all the $v_i$, $i=1, \ldots, k$, lie on $Q$ as well. This proves the claim.    
\end{example}

\bibliographystyle{alpha}
\bibliography{biblio}
\end{document}